\documentclass{amsart}
\usepackage{lineno}
\usepackage[utf8]{inputenc}
\usepackage[english]{babel}
\usepackage[margin=1.2in]{geometry}
\usepackage{amsmath, amsthm, amscd, amssymb, mathtools}
\usepackage{verbatim} % remove for comments
\usepackage{wrapfig}
\usepackage{todonotes}
\usepackage{bbm}
\usepackage{array}
\usepackage[nocompress]{cite}

\usepackage{hyperref}
\hypersetup{
  colorlinks   = true, %Colours links instead of ugly boxes
  urlcolor     = blue, %Colour for external hyperlinks
  linkcolor    = blue, %Colour of internal links
  citecolor   = red %Colour of citations
}

\usepackage{tikz-cd}
\usetikzlibrary{decorations.pathmorphing}
\tikzset{
  symbol/.style={
    draw=none,
    every to/.append style={
      edge node={node [sloped, allow upside down, auto=false]{$#1$}}}
  }
}

\usepackage{enumitem}
\setenumerate{itemsep=0pt,topsep=3pt}
\setenumerate[1]{label=\textup{(\roman*)}}

\newcommand{\PreserveBackslash}[1]{\let\temp=\\#1\let\\=\temp}
\newcolumntype{C}[1]{>{\PreserveBackslash\centering$}p{#1}<{$}}
\newcolumntype{R}[1]{>{\PreserveBackslash\raggedleft\scriptsize$}p{#1}<{$}}

\theoremstyle{plain}
\newtheorem{theorem}{Theorem}[section]
\newtheorem{lemma}[theorem]{Lemma}
\newtheorem{proposition}[theorem]{Proposition}
\newtheorem{corollary}[theorem]{Corollary}
\newtheorem{conjecture}[theorem]{Conjecture}
\newtheorem{innermaintheorem}{Theorem}
\newenvironment{maintheorem}[1]
  {\renewcommand\theinnermaintheorem{\ref{#1}}\innermaintheorem}
  {\endinnermaintheorem}

\theoremstyle{definition}
\newtheorem{definition}[theorem]{Definition}
\newtheorem{remark}[theorem]{Remark}
\newtheorem{example}[theorem]{Example}

\DeclareMathOperator{\Hom}{Hom}

\DeclareMathOperator{\Aut}{Aut}
\DeclareMathOperator{\End}{End}
\newcommand{\id}{\mathrm{id}}
\DeclareMathOperator{\im}{im}
\DeclareMathOperator{\rk}{rk}

\DeclareMathOperator{\Mat}{Mat}

\DeclareMathOperator{\Gal}{Gal}
\DeclareMathOperator{\lcm}{lcm}

\newcommand{\ent}{\mathcal{O}}
\newcommand{\A}{\mathcal{A}}

\newcommand{\mapsfrom}{\mathrel{\reflectbox{\ensuremath{\mapsto}}}}
\newcommand{\tensor}{\otimes}

\newcommand{\Z}{\ensuremath{\mathbb{Z}}}
\newcommand{\Q}{\ensuremath{\mathbb{Q}}}
\newcommand{\R}{\ensuremath{\mathbb{R}}}

\newcommand{\F}{\ensuremath{\mathbb{F}}}
\newcommand{\m}{\ensuremath{\mathfrak{m}}}

\newcommand{\clq}{C}
\newcommand{\clqb}{D}
\newcommand{\crg}{k}
\newcommand{\nrg}{R}

\newcommand{\note}[1]{\textbf{\color{blue}#1}}
\newcommand{\inote}[1]{\textbf{\color{red}Ignacio: #1}}
\title
[On exceptional cliques in matrix rings]
{On exceptional cliques in matrix rings}
\author[Milan Boutros]{Milan Boutros} 
\author[Ignacio Cascudo]{Ignacio Cascudo} % IMDEA Software Institute 
\author{Ronald Cramer} % Centrum Wiskunde en Informatica, Universiteit Leiden
\author{Dani\"el van Gent} % Centrum Wiskunde en Informatica
\author[Chaoping Xing]{Chaoping Xing} % Shanghai Jiao Tong University
\begin{document}
%\linenumbers

\begin{abstract}
	We study the notion of exceptional clique, a subset of a ring such that the difference of any two distinct elements of the subset is invertible. Motivated by applications in cryptography, our main focus is to determine the largest size of an exceptional clique in the ring $\textup{Mat}_{n\times n}(\Z)$ of square $n\times n$ matrices over the integers, for every $n$.
	We obtain several results for the question above, both in the general case and the ``commutative'' case where we additionally require that the elements in the clique commute with each other. As highlights, we prove that, at least for some values of $n$, the largest exceptional cliques in $\textup{Mat}_{n\times n}(\Z)$ are necessarily non-commutative; we then show that for an infinite family of $n$, there are non-commutative exceptional cliques of size $n^2$, and that for every $n$ there are commutative exceptional cliques of size $\frac23 n+O(n^{\theta})$, for a constant $\theta>\frac{11}{20}$.
\end{abstract}	

\maketitle

\begin{comment}
Alternative abstract: 

We study the notion of exceptional clique, a subset of a ring such that the difference of any two distinct elements of the subset is invertible. Motivated by applications in cryptography, our main focus is to determine the largest size of an exceptional clique in the ring $\textup{Mat}_n(\Z)$ of square $n\times n$ matrices over the integers, for every $n$.
	We obtain several results regarding this problem, both in the general case and in the ``commutative'' case where we additionally require that the elements in the clique commute with each other.
	In the latter case, we show that the problem reduces to studying exceptional cliques in orders in number fields, and we prove that for every $n$ there are commutative exceptional cliques of size $\frac23 n+O(n^{\theta})$ for a constant $\theta>\frac{11}{20}$. We also derive upper bounds on commutative cliques via reduction modulo primes.
	In the general case, we construct an infinite family of $n$ for which there are non-commutative exceptional cliques of size $n^2$.
	We also record the largest exceptional cliques known to us in small dimensions, both in orders and in matrix rings. 
	Finally, we define a generalized notion of exceptional cliques in $\textup{Hom}(\Z^m,\Z^n)$ and prove that when $n+1$ is a safe prime and $0<m\le n/2$, optimal cliques of size $2^n$ exist.
\end{comment}

%!TEX root = main.tex

\section{Introduction}

An \emph{exceptional clique} in a ring \(R\) is a subset \(\clq\subseteq R\) such that \(u-v\) is a unit for all distinct \(u,v\in \clq\). Exceptional cliques in the context of number fields were first considered by Lenstra \cite{Lenstra}, whence we derive the term `exceptional clique'.
He showed that if the maximal order \(\mathcal{O}_K\) of a number field \(K\) contains a sufficiently large exceptional clique, in relation to its discriminant and other parameters, then \(K\) is Euclidean.
Subsequently he discovered many large-degree Euclidean number fields, of which up until then only very few were known.

Research in this direction was continued by \cite{Martinet,Leutbecher}.
They proved upper bounds on the \emph{clique number} \(\omega(\mathcal{O}_K)\), the size of the largest exceptional clique of \(\mathcal{O}_K\), for \(K\) of small degree. They also provided a family of number fields of unbounded degree that achieve the best known asymptotic growth of \(\omega(\mathcal{O}_K)\) as a function of the degree \(n\) of \(K\), namely \(2n+O(1)\).

% \emph{rate} 

%\note{Correct term? Inconsistent with def later} \inote{if we are not using it again maybe we can just say ``achieve the best known asymptotic growth of...''}, the .
% We have compiled a table of number fields of largest known Lenstra constants in small degree in Section~BLAH.

%Another context is graph theory.
%An exceptional clique of a ring \(R\) is indeed a clique in a graph theoretic sense: it is a clique of the graph with vertex set \(R\) and edge set \(\{\{x,y\}\,|\,x-y\in R^*\}\).
%If \(R\) is a finite field, one obtains Paley graphs, which have been studied by XYZ.
%In Section~BLAH we give alternative bounds on the clique number of Paley graphs. \inote{I am not completely sure about this paragraph, because I am not aware of what the consequences are of our work for Paley graphs. If they are not too important, I would postpone even refering to it until when we collect all our contributions.}
Beyond number fields, exceptional cliques in different types of rings have found a number of applications in cryptography, particularly in the areas of secret sharing, zero knowledge proofs, and authentication. These applications benefit from having  exceptional cliques which are as large as possible in a given ring. Usually, these areas have dealt with rings for which determining the largest exceptional cliques, and constructing them explicitly, is trivial. Indeed, the most usual case by far is that of a finite field, where the full space is an exceptional clique, and the property is used tacitly. More recently some works have considered cryptographic applications which require exceptional cliques in Galois rings~\cite{ACDEY19, DLS20, ACCDE22} or more generally, in rings of the form $(\Z/m\Z)[X]/(f(X))$ \cite{ACK21, AL21, BCFK21, CCKP22, GNS23, BC25}.\footnote{The cryptography literature has not been consistent in naming for exceptional cliques: in addition to ``exceptional sets'' or ``exceptional sequences'', they have received names such as ``subtractive sets'' \cite{AL21}, ``admissible sets'' \cite{CDN15}, ``sampling sets'' \cite{BCFK21, CCKP22} or ``strongly sampling sets''\cite{BC25}.} Determining the largest exceptional cliques in these rings is still trivial.\footnote{The largest exceptional clique of a Galois ring is given by its residue field, while rings of the form $(\Z/m\Z)[X]/(f(X))$ can be written as direct product of Galois rings and fields and the size of the largest clique is given by the minimum of those of the factors.}

%required to use exceptional cliques in Galois rings~\cite{ACDEY19, DLS20, ACCDE22}, where the largest clique is given by their residue field; or more generally, in rings of the form $(\Z/m\Z)[X]/(f(X))$ \cite{ACK21, AL21, BCFK21, CCKP22, BC25}, commonplace in lattice-based cryptography, and which can be written as direct product of Galois rings and fields; while works such as \cite{GNS23} have considered cryptographic tools that work over any commutative ring as long as it has a large enough exceptional clique. 

%The underlying property used in these works is that, for a commutative ring $R$, interpolation on any subset of $d+1$ elements of an exceptional clique induces a $R$-module isomorphism between the set of polynomials in $R[X]$ of degree at most $d$, and $R^{d+1}$ \cite[Theorem 11.124]{CDN15}. Aside from the more abstract use from \cite{GNS23}, all the aforementioned works have considered rings where maximal exceptional cliques are easy to determine.

However, recently some applications have arisen which have motivated the question of what is the largest exceptional clique in $\textup{Mat}_n(\Z)$, the ring of square $n\times n$ matrices over the integers. In particular, these applications deal with cryptographic tools such as black-box secret sharing~\cite{CF02, CFS05} and zero-knowledge proofs for hidden order groups \cite{CD09, BC24}. In Section~\ref{sec:exceptional_cliques_in_cryptography} we give more details on the history of these applications in cryptography.

The main goal of this work is to study the clique number $\omega(\textup{Mat}_n(\Z))$ and in particular
to give lower bounds for this number and produce explicit large exceptional cliques in \(\textup{Mat}_n(\Z)\). Moreover, we also study similar questions for ``commutative'' exceptional cliques, those cliques in \(\textup{Mat}_n(\Z)\) with the additional restriction that any two elements of the clique commute with each other. We denote the size of the largest exceptional clique of this type by $\omega_{\textup{com}}(\textup{Mat}_n(\Z))$.\\

We now summarize some of our results. 

We first focus on the aforementioned commutative case. Clearly, every order of rank \(n\) embeds in \(\textup{Mat}_n(\Z)\), and so do its cliques. In Proposition~\ref{prop:commutative_clique_to_nf} we show that each  \textit{commutative} clique in \(\textup{Mat}_n(\Z)\) can be embedded in an order $\mathcal{O}_K$ of rank at most \(n\). Upper bounds on the largest size of cliques in these orders will therefore imply upper bounds on $\omega_{\textup{com}}(\textup{Mat}_n(\Z))$.

There is a simple upper bound of \(\omega(\mathcal{O}_K)\leq 2^n\) for each number field \(K\) of degree \(n\geq 1\) \cite{Lenstra}. This bound is not sharp and \cite{Martinet} shows that \(\omega(\mathcal{O}_K)\leq 7 <2^3\) for every number field \(K\) of degree $3$ (and the bound is achieved), and consequently $\omega_{\textup{com}}(\textup{Mat}_3(\Z))=7$. We add to this with the following result.

\begin{maintheorem}{thm:computational_commutative}
	Exceptional cliques in orders of rank \(4\) and \(5\) have respectively at most \(10\) and \(11\) elements. These bounds are sharp.
\end{maintheorem}

As far as we know, however, there are no known stronger asymptotic bounds. Meanwhile, the lower bound $\omega(\mathcal{O}_K)\geq 2n+O(1)$ is very far from the upper bound $\omega(\mathcal{O}_K)\leq 2^n$. Moreover, this linear lower bound does not apply to every $n$ but only a concrete infinite family of integers $n$: concretely, we know that for $n$ of the form $(p-1)/2$, with $p$ an odd prime, $\Z[\zeta_p+\zeta_p^{-1}]$ where $\zeta_p$ is a primitive $p$-th root of unity, has rank $n$ and cliques of size $2n$; similarly, a worse, but still linear, bound applies to $\Z[\zeta_p]$, which has rank $n=p-1$ and cliques of size $n$. These facts do not immediately provide, in principle, guarantees that the bound applies uniformly to every (large enough) $n>0$; indeed it is important to note (Lemma~\ref{lem:matrix_embed}) that there is a ring homomorphism \(\textup{Mat}_m(\Z)\to\textup{Mat}_n(\Z)\) \emph{only if} \(m\mid n\), and thus a clique in \(\textup{Mat}_n(\Z)\) generally does not naturally embed into \(\textup{Mat}_{n+1}(\Z)\); similarly, it is generally not true that \(C^2\subseteq R^2\) is an exceptional clique when \(C\subseteq R\) is an exceptional clique; therefore it does not seem easy to determine whether \(\omega(\textup{Mat}_n(\Z))\) or $\omega_{\textup{com}}(\textup{Mat}_n(\Z))$ are monotonically increasing.

Faced with this state of affairs, we show the following lower bound for commutative cliques, which can be derived from an analytic number theory result by Matom\"aki, Maynard and Shao~\cite{Vinogradov}, and yields a linear lower bound which holds for \textit{every} $n>0$. 
%gives an easy improvement of the first kind.

\begin{maintheorem}{thm:goldbach}
Fix \(\theta > \frac{11}{20}\). For all \(n\geq 0\) there exists a commutative exceptional clique in \(\textup{Mat}_n(\Z)\) of size \(\frac{2}{3}n + O(n^{\theta})\).
\end{maintheorem}

We now turn to the non-commutative case, which gives arguably the most interesting results. We first remark that the general upper bound  $\omega(\Mat_n(\Z))\leq 2^n$ still applies for non-commutative cliques, see Corollary~\ref{cor:clique_bound}.

The first interesting fact is that this bound is \emph{achievable} for all $n\leq 4$, in particular,\footnote{We note that for $n\leq 2$,  $\omega(\Mat_n(\Z))=\omega_{\textup{com}}(\Mat_n(\Z))=2^n$ } $\omega(\Mat_3(\Z))=8$ and $\omega(\Mat_4(\Z))=16$. The former  was shown in \cite{BC24} and the latter was found shortly thereafter in unpublished work, in both cases computationally; here we contribute a more conceptual construction for cliques attaining these bounds (Example~\ref{ex:n4c16} and Example~\ref{ex:n3c8}).

Therefore, taking into account the upper bounds for the commutative case above, $\omega_{\textup{com}}(\Mat_3(\Z))\leq 7$ and $\omega_{\textup{com}}(\Mat_4(\Z))\leq 11$, this means that $\omega(\Mat_n(\Z))> \omega_{\textup{com}}(\Mat_n(\Z))$ for at least some values $n$. Motivated by this, we explore lower bounds for $\omega(\Mat_n(\Z))$ for diverse families of $n$, and make constructions that can beat the best lower bounds we know for $\omega_{\textup{com}}(\Mat_n(\Z))$. 

Using graph theory, we present a lower bound that is still linear but beats the known bounds in the commutative case for $n=p-1,p$ with $p$ an odd prime.

\begin{maintheorem}{thm:graphs}
For all odd primes \(p\) the rings \(\textup{Mat}_{p-1}(\Z)\) and \(\textup{Mat}_p(\Z)\) both contain an exceptional clique consisting of \(p+1\) symmetric matrices with coefficients in \(\{-2,-1,0,1\}\).
\end{maintheorem}
For $n=p-1$ this yields $\omega(\Mat_n(\Z))\geq n+2$, which is only slightly larger than the known bound $\omega_{\textup{com}}(\Mat_n(\Z))\geq n+1$. But for $n=p$ it is a real improvement, as the best we know in the commutative case is  $\omega_{\textup{com}}(\Mat_n(\Z))\geq \frac{2}{3}n + O(n^{\theta})$ (see above).

However, somewhat surprisingly, we show that we can find superlinear lower bounds for some infinite families of $n$. Our main result is the following theorem, which relies on number theory.

\begin{maintheorem}{th:main}
Let \(p,q\) be distinct odd primes such that \(p\not\equiv \pm 1\ (\textup{mod }q)\).
Then there exists an exceptional clique of size $(p-1)^2$ in $\Mat_{(q-1)(p-1)/4}(\Z)$.
\end{maintheorem}

If we take \(q=5\) in this theorem, we obtain \(\omega(\textup{Mat}_{p-1}(\Z))\geq (p-1)^2\) for all primes \(p\not\equiv \pm 1\ (\textup{mod }5)\). 
The best known cliques in \(\textup{Mat}_{n}(\Z)\) thus attain a quadratic rate, which is a significant improvement over the linear rate in orders.
It is not unlikely that, analogously to Theorem~\ref{thm:goldbach}, we can obtain a quadratic rate for all \(n\), not just \(n\) of the form \(p-1\), although with a worse constant (Remark~\ref{rem:goldbach}).
%In Section~\ref{} we argue heuristically that \(\omega(\textup{Mat}_n(\Z))=O(n^2)\).
It is interesting to note that the best known rate is again proportional to the rank of the ring seen as a \(\Z\)-module.

We have made some attempts to improve the results of Theorem~\ref{th:main}.
Firstly, the cliques in Theorem~\ref{th:main} are constructed from cyclic algebras, and we may generalize the construction to include twisted cyclic algebras (Proposition~\ref{prop:nt_square_clique_twisted}).
Although we were unable to produce asymptotic improvements, this construction produces the best known cliques in \(\textup{Mat}_n(\Z)\) for some values of \(n\). This is apparent in Section~\ref{sec:table_of_records}, where we present a table of best known cliques in \(\textup{Mat}_n(\Z)\) for small \(n\).
Secondly, we attempted recursive constructions of cliques.
Using coding theory we may turn a clique in \(\textup{Mat}_m(\Z)\) of size a prime power \(q\) into a clique of size \(q^2\) in \(\textup{Mat}_{n}(\Z)\) for \(n\geq m(q+1)/2\).
However, recursive application yields cliques of size only \(\Theta(n\log n)\) in \(\textup{Mat}_n(\Z)\) (Theorem~\ref{thm:coding}).
Another attempt at a recursive construction involves spreads (Theorem~\ref{thm:clique_doubling}).
Although the construction cannot in fact be applied recursively, it also produces some best known cliques for small \(n\).

We finish this summary by highlighting two additional results in our work: first, for small values of \(n\), it is possible to enumerate all number fields \(K\) of degree \(n\) for which \(\omega(\mathcal{O}_K)\) is maximal.
Up to a natural notion of equivalence of exceptional cliques (Definition~\ref{def:clique_equivalence}) we can even enumerate all maximal cliques (e.g., Example~\ref{ex:phi}).
In Section~\ref{sec:rank_2} we do the same for \(\textup{Mat}_2(\Z)\).

Finally, we consider a generalized notion of an exceptional clique, which also has applications in cryptography (see Section~\ref{sec:exceptional_cliques_in_cryptography}). An \emph{exceptional clique} in \(\Hom(\Z^m,\Z^n)\) is a subset \(\clq\) such that for any two distinct \(f,g\in \clq\) the map \(f-g\) has a left inverse. \footnote{In other words, it is a set of integer matrices of dimensions $n\times m$ where the difference of any two distinct matrices $M$, $N$ in the set has a left inverse, i.e. there exists a $m\times n$ integer matrix $A$ with $A(M-N)=I_{m}$ where $I_m$ is the $m\times m$ identity matrix.} If \(m=n\), then \(\Hom(\Z^m,\Z^n)\cong\textup{Mat}_n(\Z)\) and the two definitions agree. However, for some values $m<n$ this generalized notion of an exceptional clique achieves the optimal size $2^n$. Recall that a prime number \(p\) is \emph{safe} if \(p\) is odd and \((p-1)/2\) is prime.

\begin{maintheorem}{thm:rank-metric}
For every safe prime \(n+1\) and \(0<m\leq n/2\) there exists an exceptional clique in \(\Hom(\Z^m,\Z^n)\) of size \(2^n\).
\end{maintheorem}

It is conjectured that there are infinitely many safe primes, see for example \cite{Shoup}, which would make the result above hold for infinitely many $n$.\\

We give an overview of our sections.
Section~\ref{sec:preliminaries} consists of preliminaries and basic observations. In Sections~\ref{sec:commutative},~\ref{sec:uppercommutative} we focus on commutative cliques: in Section~\ref{sec:commutative} we first prove results on small commutative cliques, then show the aforementioned asymptotic lower bound (Theorem \ref{thm:goldbach}) and then prove that every commutative clique arises from an order; based on the latter, in Section~\ref{sec:uppercommutative} we show several upper bounds for the size of commutative cliques.
Section~\ref{sec:spreads} uses the concept of spreads which not only allows us to provide a simpler proof of the aforementioned \(2n+O(1)\) bound for commutative cliques from \cite{Leutbecher} but also leads to a lower bound for the size of non-commutative cliques in terms of those of commutative cliques, in certain dimensions.
In Sections~\ref{sec:cyclicalgebras},~\ref{sec:coding},~\ref{sec:graph} we prove asymptotic lower bounds for the size of (not necessarily commutative) exceptional cliques based respectively on results from cyclic algebras, coding theory and graph theory. In Section~\ref{sec:rank_2} we classify all maximal cliques in \(\textup{Mat}_2(\Z)\) up to conjugation. In Section~\ref{sec:miscellanea} we show some additional results that do not fit in the previous sections. In Section~\ref{sec:table_of_records} we present tables with the largest cliques known to us in small dimensions, both in orders and in matrix rings. In Section~\ref{sec:rankmetric} we consider the more general notion of an exceptional clique in \(\textup{Hom}(\Z^m,\Z^n)\). Finally in Section~\ref{sec:exceptional_cliques_in_cryptography} we provide an overview of the applications of these notions in cryptography.

%\input{introduction1}

% !TEX root = ../main.tex

\section{Preliminaries}\label{sec:preliminaries}

%\begin{definition}
Throughout this work, rings are assumed to be unital. Let \(\nrg\) be a ring.

\begin{definition}
An \emph{exceptional clique} in \(\nrg\) is a subset \(\clq\subseteq \nrg\) such that for all distinct \(u,v\in \clq\) we have \(u-v\in \nrg^*\). We say an exceptional clique is \emph{commutative} if its elements commute pairwise. We denote \(\omega(R):=\sup\{\#C: C\subseteq R \textrm{ is an exceptional clique} \}\) and \(\omega_{\textup{com}}(\nrg)\) the same supremum restricted to commutative cliques.
An \emph{exceptional unit} of \(\nrg\) is an element \(u\in \nrg^*\) such that \(1-u\in \nrg^*\).
\end{definition}

%The \emph{Lenstra constant} \(L(\nrg)\) of \(\nrg\) is the maximal cardinality of an exceptional clique in \(\nrg\). For \(n\geq0\) we abbreviate \(L(\textup{Mat}_n(\Z))\) as \(L(n)\).
%\end{definition}

\begin{proposition}\label{prop:injective_quotient}
	Let \(\nrg\) be a ring and \(\clq\subseteq \nrg\) an exceptional clique. Then
	\begin{enumerate}[nosep]
		\item For every non-zero left-\(\nrg\)-module homomorphism \(\nrg\to M\) the restriction to \(\clq\) is injective.
		\item For every ring homomorphism \(\nrg\to S\) the image of \(\clq\) is an exceptional clique.
	\end{enumerate}
\end{proposition}
Note that if \(f:\nrg\to S\) is a non-zero ring homomorphism, then \(S\) is a left-\(\nrg\)-module and \(f\) is an \(\nrg\)-module homomorphism to which part~(i) of the proposition also applies.
\begin{proof}
	Let \(f:\nrg\to M\) be an \(\nrg\)-module homomorphism. Suppose \(f(x)=f(y)\) for distinct \(x,y\in \clq\). Then \(f(x-y)=0\) where \(x-y\in \nrg^*\). Hence \(f=0\), proving (i). For (ii), note that ring homomorphisms preserve units.
\end{proof}

\begin{corollary}\label{cor:clique_bound}
	Let \(\crg\) be a commutative ring and let \(n\in\Z_{\geq 0}\). Then
	\begin{enumerate}[nosep]
		\item if \(\crg\) is a finite field, then \(\omega(\textup{Mat}_n(\crg))\leq\#\crg^n\).
		\item if \(\crg=\Z\), then \(\omega(\textup{Mat}_n(\crg))\leq 2^n\).
	\end{enumerate}
\end{corollary}
\begin{proof}
	Apply part (i) and (ii) of Proposition~\ref{prop:injective_quotient} to \(\textup{Mat}_n(\crg)\to \crg^n\), given by \(M\mapsto Mv\) for any fixed non-zero \(v\in \crg^n\), and \(\textup{Mat}_n(\Z)\to \textup{Mat}_n(\Z/2\Z)\) respectively.
\end{proof}

An \emph{order} is an integral domain whose additive group is free of finite rank. By the corollary, an exceptional clique in an order of rank \(n\) has size at most \(2^n\).
For an exceptional clique \(\clq\subseteq\textup{Mat}_n(\Z)\) we define its \emph{rate} to be \(\#\clq/n\).

\begin{lemma}\label{lem:trivial_induction}
	Let \(\nrg\) and \(S\) be rings, \(\crg\) a commutative ring, and \(m,n\in\Z_{\geq0}\). Then
	\begin{enumerate}[nosep]
		\item \(\omega(\nrg\times S)\geq \min\{\omega(\nrg),\omega(S)\}\).
		
		\item \(\omega(\textup{Mat}_{m+n}(\crg))\geq \min\{\omega(\textup{Mat}_m(\crg)),\omega(\textup{Mat}_n(\crg))\}\).
	\end{enumerate}
	Both statements are also true when we replace \(\omega\) by \(\omega_{\textup{com}}\).
\end{lemma}
\begin{proof}
	For (i), if \(\clq\subseteq \nrg\) and \(\clqb\subseteq S\) are exceptional cliques of the same size and $f\colon \clq\to\clqb$ is a bijection, consider the exceptional clique $\{(x,f(x)) \mid x\in\clq\}\subset \nrg\times S$. For (ii), we combine (i) with the natural ring homomorphism \(\textup{Mat}_m(\crg)\times\textup{Mat}_n(\crg)\to\textup{Mat}_{m+n}(\crg)\) and Proposition~\ref{prop:injective_quotient}. The final statement follows trivially.
\end{proof}

%\begin{remark}
One might assume that if \(\clq\subseteq \nrg\) and \(\clqb\subseteq S\) are exceptional cliques, then \(\clq\times \clqb\subseteq \nrg\times S\) is an exceptional clique. However, \((x,u)-(x,v)\) is not a unit for \(x\in \clq\) and distinct \(u,v\in \clqb\). As the following lemma shows, we cannot naturally embed \(\textup{Mat}_n(\Z)\) into \(\textup{Mat}_{n+1}(\Z)\), and thus cannot expect a clique in the former to naturally embed into the latter.

\begin{lemma}\label{lem:matrix_embed}
	For \(m,n\geq 0\), there is a ring homomorphism \(\textup{Mat}_m(\Z)\to\textup{Mat}_n(\Z)\) if and only if \(m\mid n\).
\end{lemma}
\begin{proof}
	If \(m\mid n\), then we have a morphism \(\textup{Mat}_m(\Z)\to(\textup{Mat}_m(\Z))^{n/m}\to\textup{Mat}_n(\Z)\). Conversely, suppose a morphism exists.
	Let \(f\in\Z[X]\) be monic irreducible of degree \(m\), say \(X^m-2\), and let \(M\) be its companion matrix. The image of \(M\) in \(\textup{Mat}_n(\Z)\) is a root of \(f\), so its characteristic polynomial divides some power of \(f\), and must be equal to some power of \(f\) by irreducibility.
	Hence \(m\mid n\).
\end{proof}

%\end{remark}

%\begin{definition}
%Let \(R\) be a ring. We write \(R^\textup{op}\) for the \emph{opposite ring} of \(R\), which is the ring where multiplication is given by \((x,y)\mapsto y\cdot x\). We write \(Z(R)=\{x\in R \mid (\forall y \in R)\ xy=yx\}\) for the \emph{center} of \(R\). We write \(R^\textup{en}=R\tensor_{Z(R)} R^\textup{op}\) for the \emph{enveloping ring} of \(R\), and we equip \(R\) with a natural \(R^\textup{en}\)-module structure given by \((a\tensor b) \cdot x \mapsto axb\). Note that if \(R\) is commutative, all objects constructed are simply \(R\). We define the \emph{affine group} of \(R\) to be \(\textup{Aff}(R)=R\rtimes (R^\textup{en})^*\).
%\end{definition}

We may now briefly study some symmetries of the set of exceptional cliques of $\nrg$. There are three natural group actions on $\nrg$: translation by elements of $\nrg$, left-multiplication by elements of $\nrg^*$, and action by automorphisms in $\Aut(\nrg)$. They induce actions on subsets of $\nrg$ and one easily checks that they preserve the exceptional clique property. We combine these three actions into one.

\begin{lemma}\label{lem:equivalent_clique}
The group \(G(\nrg)=\nrg\rtimes \nrg^* \rtimes\Aut(\nrg)\) acts on $R$ through translation, left-multiplication, automorphism, and preserves exceptional cliques. \qed
\end{lemma}

If $\nrg$ is non-commutative, right-multiplication by elements in $\nrg^*$ is different from left-multiplication, but this action gets accounted for through left-multiplication combined with interior automorphisms.

\begin{definition}\label{def:clique_equivalence}
	We say two cliques are \emph{equivalent} if they are in the same orbit under \(\textup{G}(\nrg)\). 
	We say exceptional units \(u\) and \(v\) are \emph{equivalent} if the cliques \(\{0,1,u\}\) and \(\{0,1,v\}\) are equivalent.
	We say an exceptional clique is \emph{normalized} if it contains \(0\) and \(1\).
\end{definition}

One check easily that any exceptional unit \(u\) is equivalent to \(u^{-1}\) and \(1-u\). Moreover, an exceptional clique can always be normalized using only translation and multiplication: pick $x\in\clq$ and set $\clq'=\clq-x$, which contains $0$, then pick $x'\in\clq'$ (which must be invertible) and set $\clq''={x'}^{-1}\cdot\clq$, which contains $0$ and $1$.

%\begin{lemma}
%Let \(R=\textup{Mat}_n(k)\) for any commutative ring \(k\) and \(n\in\Z_{\geq0}\).
%Then \(T=\sum_{i,j} E_{ij}\tensor E_{ij} \in R^\textup{en}\) satisfies \(T^2=1\) and acts by transposition on \(R\), where \(E_{ij}\) is the matrix consisting of only zeros but a one at entry \((i,j)\). \qed
%\end{lemma}

\section{Commutative cliques}\label{sec:commutative}

By Corollary~\ref{cor:clique_bound}, an exceptional clique in a number ring of rank \(n\) over \(\Z\) is bounded in size by \(2^n\). For \(n=1\) this bound is achieved by \(\{0,1\}\), while for \(n=2\) we have the following cliques.

\begin{example}\label{ex:phi}
	Suppose \(f=X^2+aX+b\in\Z[X]\) is the minimal polynomial of some exceptional unit~\(u\).
	Note that \(b\in\{\pm1\}\), being the norm of the unit \(u\).
	Similarly \(f(X+1)=X^2+(a+2)X+(a+b+1)\) is the minimal polynomial of \(u-1\), so \(a+b+1\in\{\pm1\}\).
	Hence there are exactly \(4\) different possible \(f\):
	\[ f_\zeta = X^2-X+1, \quad f_{\phi-1}=X^2+X-1,  \quad  f_\phi=X^2-X-1, \quad\text{and}\quad f_{\phi+1}=X^2-3X+1. \]
	The first defines the order \(\Z[\zeta]\), where \(\zeta\) is a primitive \(6\)-th root of unity, while the others all have their roots in \(\Z[\phi]=\Z[X]/f_\phi\).
	Note that \(\zeta-\zeta^{-1}\) is not a unit, so \(\{0,1,\zeta\}\) and \(\{0,1,\zeta^{-1}\}\) are the largest exceptional cliques in \(\Z[\zeta]\).
	They are equivalent, as \(x\mapsto \zeta^{-1}x\) maps the former to the latter.
	The exceptional units in \(\Z[\phi]\) are all equivalent, as the following diagram shows:
	\begin{center}
		\phantom{AAAAAAAAAAAAAAA}
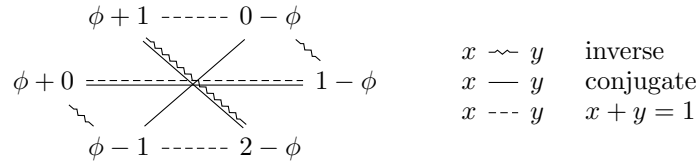
		\begin{tikzcd}[column sep={1cm,between origins}, row sep={1.732050808cm/2,between origins}]
			& \phi+1 \arrow[drdr,no head,squiggly,shift left=.03cm] \arrow[drdr,no head,shift right=.03cm] \arrow[rr,no head,dashed]  && 0-\phi \arrow[dr,no head,squiggly] \arrow[dldl,no head]  &  \\
			\phi+0 \arrow[dr,no head,squiggly] \arrow[rrrr,no head,dashed,shift left=.03cm] \arrow[rrrr,no head,shift right=.03cm] &  &&  & 1-\phi \\
			& \phi-1 \arrow[rr,no head,dashed] && 2-\phi  & 
		\end{tikzcd} \quad\quad
		\begin{tikzcd}[column sep={.9cm,between origins}, row sep={.4cm,between origins}]
			x \arrow[r,no head,squiggly]& y && \textup{\makebox[\widthof{something long}][l]{inverse}} \\
			x \arrow[r,no head]& y && \textup{\makebox[\widthof{something long}][l]{conjugate}} \\
			x \arrow[r,no head,dashed]& y && \textup{\makebox[\widthof{something long}][l]{\(x+y=1\)}}
		\end{tikzcd}
	\end{center}
	Note that \((\phi+1)-\phi\) is a unit, so \(\{0,1,\phi,\phi+1\}\) is an exceptional clique of size \(4\).
	In particular, it reaches the upper bound of Corollary~\ref{cor:clique_bound}. 
	We leave it as an exercise to the reader to show that there are \(6\) exceptional cliques of size \(4\) containing \(0\) and \(1\), and that under the action of the three operations in the diagram all these cliques are equivalent.
\end{example}

The order \(\Z[\phi]\) is involved in many large known commutative cliques in small dimensions.

\begin{example}\label{ex:nf_deg_4}
	Let \(\phi\) be a root of \(X^2-X-1\) and let \(\alpha\) be a root of \(X^2-\phi X-1\). Then 
	%\[\{0,1,\phi\}\cup\big(\alpha\phi-\{0,1,\phi+1\}\big) \cup \big(\alpha(\phi-1)-\{0,2-\phi\} \big) \cup \big(\alpha-\{1,\phi+1\}\big)\subseteq \Z[\alpha]\]
	\[\{0,1,-\phi,\alpha-1,\alpha-\phi,\alpha-\phi+1,\phi(\alpha-1),\phi(\alpha-1)-1,(\phi-1)(\alpha-1),(\phi-1)\alpha-1\} \subseteq \Z[\alpha]\]
	is an exceptional clique of size \(10\) in an order of rank \(4\). Note that the corresponding field is not Galois over \(\Q\), let alone abelian.
\end{example}

The Galois action does not interact too nicely with exceptional cliques, as expressed by the following proposition.

\begin{proposition}\label{prop:no_galois_clique}
	If \(K\) is a finite Galois extension of \(\Q\) and the Galois orbit of \(\alpha\in\mathcal{O}_K\) is an exceptional clique, then \(\alpha\in\Z\).
\end{proposition}
\begin{proof}
	The discriminant of \(\Z[\alpha]\) is given by \(\prod_{i<j} (\alpha_i-\alpha_j)^2\), where \(\alpha_1,\dotsc,\alpha_n\) are the conjugates of \(\alpha\). By assumption it is a unit, so \(\Q(\alpha)=\Q\) by Minkowski's theorem.
\end{proof}

%\begin{example}\label{ex:nf_deg_6}
%Let \(\phi\) be a root of \(X^2-X-1\) and write \(\phi_1 = \phi+1\) and \(\phi_2=2\phi+1\).
%Let \(\alpha\) be a root of \(X^3-\phi X^2-2\phi_1X+\phi_2\). 
%Then 
%\begin{align*}
%&\{0, 1, -\phi^{-1} b^{2} - b - \phi, -\phi^{-1} b^{2} -\phi_1^{-1} b - \phi_1, -\phi^{-1} b^{2} - b - 2 \phi, -\phi_1^{-1} b^{2} -\phi^{-1} b + \phi + 2, \\ 
%&-2\phi^{-1} b^{2} - b - 4 \phi, -\phi^{-1} b^{2} - b +\phi^{-1}, -\phi^{-1} b^{2} - \phi, \phi_2^{-1} b^{2} + \phi_2^{-1} b - \phi, \phi^{-1} b +\phi^{-1}, \\ 
%& 2\phi_2^{-1} b^{2} -\phi_1^{-1} b - 4 \phi + 5, \phi_1^{-1} b^{2} -\phi_1^{-1} b - \phi, -\phi^{-1} b^{2} - 2 \phi + 1, \phi_2^{-1} b^{2} -\phi_1^{-1} b +\phi^{-1}, \\ 
%& \phi_2^{-1} b^{2} + -\phi_1^{-1} b + 2\phi^{-1}, \phi_2^{-1} b^{2} +\phi^{-1}, \left(-6 \phi + 10\right) b^{2} - \phi_2^{-1} b + 5 \phi - 9\} \subseteq \Z[\alpha]
%\end{align*}
%is an exceptional clique of size \(18\) in an order of rank \(6\).
%\end{example}

The following theorem by Leutbecher and Niklasch exhibits a family of orders, which includes the order \(\Z[\phi]\), with exceptional cliques that achieve the highest rate known in the number field case, namely \(2\).

\begin{theorem}[Theorem~3 in \cite{Leutbecher}]\label{thm:nt_solution}
	Let \(p\) be an odd prime. Then \(\Z[\zeta_p+\zeta_p^{-1}]\) has rank \((p-1)/2\) and contains an exceptional clique of size \(p-1\). \qed
\end{theorem}

Theorem~\ref{thm:nt_solution} gives a commutative exceptional clique in \(\textup{Mat}_n(\Z)\) for a set of \(n\) of natural density~\(0\). At the cost of a worse rate of~\(\frac{2}{3}\), we move to density \(1\) using a result from Matom\"aki, Maynard and Shao.

\begin{theorem}\label{thm:goldbach}
	Fix \(\theta > \frac{11}{20}\). For all \(n\geq 0\) there exists a commutative exceptional clique in \(\textup{Mat}_n(\Z)\) of size \(\frac{2}{3}n + O(n^{\theta})\).
\end{theorem}
\begin{proof}
	We may assume \(n\) is `sufficiently large', since for small \(n\) we may simply take the empty clique. 
	Then the number \(2n+3\) can be written as a sum of three odd primes \(p_1\), \(p_2\) and \(p_3\) of size \(\frac{2}{3}n+O(n^\theta)\) by Theorem 1.1 in \cite{Vinogradov}.
	With \(n_i=(p_i-1)/2\), the ring \(\textup{Mat}_{n_i}(\Z)\) contains a commutative exceptional clique of size \(p_i-1\) by Theorem~\ref{thm:nt_solution}.
	As \(n=n_1+n_2+n_3\), the ring \(\textup{Mat}_{n}(\Z)\) contains by Lemma~\ref{lem:trivial_induction} a commutative exceptional clique of size \(\min\{p_1,p_2,p_3\}-1=\frac{2}{3}n+O(n^\theta)\).
\end{proof}

The cliques constructed in this theorem are commutative because they are contained in a product of orders.
We will show conversely that every commutative clique arises from an order. 

\begin{definition}
	For a ring \(\nrg\) let \(\textup{Jac}(\nrg)\) denote the \emph{Jacobson radical} of \(\nrg\), which is the subset of elements that annihilate all simple left \(\nrg\)-modules.
	An algebra \(A\) over a field \(K\) is \emph{semisimple} if \(\textup{Jac}(A)\) is trivial, and \emph{separable} if it is finite dimensional and for every field extension \(K\to L\) the \(L\)-algebra \(L\tensor_K A\) is semisimple.
\end{definition}

\begin{theorem}[Artin--Wedderburn, p.\ 69 in \cite{wedderburn_book}]\label{thm:artin-wedderburn}
	Let \(K\) be a field and \(A\) a semisimple \(K\)-algebra. Then there exist integers \(0\leq n\) and \(0<m_1,\dotsc,m_n\) and division \(K\)-algebras \(D_1,\dotsc,D_n\) such that  \(A\cong\textup{Mat}_{m_1}(D_1)\times\dotsm\times\textup{Mat}_{m_n}(D_{n})\), and this representation is unique up to reordering of the indices. \qed
\end{theorem}

\begin{theorem}[Wedderburn Principal Theorem, p.\ 143 in \cite{wedderburn_book}]\label{thm:wedderburn}
	Suppose \(K\) is a field and \(A\) a finite-dimensional (not necessarily commutative) \(K\)-algebra with Jacobson radical \(J\). Then \(A/J\) is a separable \(K\)-algebra if and only if there exists a separable \(K\)-subalgebra \(B\subseteq A\) such that \(A=B \oplus J\). \qed
\end{theorem}

\begin{lemma}\label{lem:dimension_bound}
	Let \(K\) be a field and \(n\in\Z_{\geq0}\). 
	If \(A\subseteq\textup{Mat}_n(K)\) is a separable commutative \(K\)-subalgebra, then \(\dim_K(A)\leq n\).
\end{lemma}
\begin{proof}
	By definition \(L\tensor_K A\) is separable over \(L\), where \(L\) is an algebraic closure of \(K\).
	As \(\dim_{L}(L\tensor_KA)=\dim_K(A)\), we may assume without loss of generality that \(K\) is algebraically closed. 
	Then by Theorem~\ref{thm:artin-wedderburn} we obtain that \(A\cong K^m\) as \(K\)-algebra for some \(m\geq 0\).
	This decomposition induces a decomposition of the \(A\)-module \(K^n\) into \(m\) \(K\)-submodules. Since \(A\) acts faithfully, each of these submodules is of dimension at least \(1\). Hence \(\dim_K(A)=m\leq n\).
\end{proof}

Note that the separability condition in the lemma is required. For \(n>1\) the subset \(\{ x+ \big(\begin{smallmatrix} 0 & M \\ 0 & 0 \end{smallmatrix}\big)\,|\, x\in K,\, M\in\textup{Mat}_n(K)\}\subseteq\textup{Mat}_{2n}(K)\) is in fact a commutative \(K\)-subalgebra of dimension \(1+n^2>2n\).

\begin{proposition}\label{prop:commutative_clique_to_nf}
	Let \(R\) be a domain with perfect field of fractions \(K\).
	If there exists an exceptional clique in \(\textup{Mat}_n(R)\) of cardinality \(m\) whose elements pairwise commute, then there exists a field extension \(L\) of \(K\) of degree at most \(n\) such that \(L\) contains an exceptional clique of cardinality \(m\) whose elements are integral over \(R\).
\end{proposition}
\begin{proof}
	Let \(A\) be the (commutative) \(K\)-algebra generated by the exceptional clique \(\clq\), let \(\mathfrak{m}\subseteq A\) be a maximal ideal and write \(L=A/\mathfrak{m}\), which is a field.
	%Note that the image of \(\clq\) in \(\ell\) is an exceptional clique, since ring homomorphisms preserve units. Moreover, this mapping is injective: If \(u,v\in \clq\) have the same image in \(\ell\), then \(u-v\in\mathfrak{m}\) is not a unit, hence \(u=v\). 
	By Proposition~\ref{prop:injective_quotient}, the image of \(\clq\) in \(L\) is an exceptional clique of cardinality \(m\).
	Each element of \(\textup{Mat}_n(R)\) is integral over \(R\) by Cayley--Hamilton, so the same holds for the image of \(\clq\) in \(L\).
	We have inclusions \(L \to A/\textup{Jac}(A)\to A \to \textup{Mat}_n(K)\), where the first is the Chinese remainder theorem applied to the (finite) set of maximal ideals of \(A\), and the second is Theorem~\ref{thm:wedderburn} combined with the fact that \(A/\textup{Jac}(A)\) is automatically separable because \(K\) is perfect. 
	Hence \(A/\textup{Jac}(A)\) and consequently \(L\) are of dimension at most \(n\) over \(K\) by Lemma~\ref{lem:dimension_bound}.
\end{proof}

\section{Upper bounds on commutative cliques}\label{sec:uppercommutative}

In this section we will give upper bounds on the clique number for orders, both for small rank and in general. We do this by studying cliques in \'etale algebras.

For \(n=1,2\), the upper bound of \(2^n\) from Corollary~\ref{cor:clique_bound} is achieved by the orders \(\Z\) and \(\Z[\phi]\), see Example~\ref{ex:phi}.
For \(n=3\), Leutbecher and Martinet (Theorem~4.1.1 in \cite{Martinet}) proved for orders that the largest exceptional clique has cardinality \(7=2^n-1\) and that this is achieved only for \(\Z[\zeta_7+\zeta_7^{-1}]\).
For \(n=4\) they also proved for non-totally real number fields that the largest exceptional clique that can occur has size \(9\) (Theorem~5.1.1 in \cite{Martinet}), and can be found only in \(\Z[\phi][X]/(X^2-X-\phi)\). In the totally-real case we have a lower bound of \(10\), see Example~\ref{ex:nf_deg_4}, and we will show that this is maximal as well.
The argument of Leutbecher and Martinet for the upper bound relies on the unit rank of the number field being small, effectively using the primes at infinity. Our argument involves the finite primes.

\begin{definition}\label{def:M}
	Let \(K\) be a field. 
	If \(A\) is a commutative \(K\)-algebra of finite dimension, we write \(M(A/K)\) for the maximal cardinality of a subset \(\clq\subseteq A\) such that \(N_{A/K}(x-y)=\pm 1\) for all distinct \(x,y\in \clq\). If \(n\in\Z_{\geq 0}\) we write \(M(n,K)\) for the maximum of \(M(A/K)\), where \(A\) ranges over all commutative \(K\)-algebras of dimension $n$.
	%We analogously define \(M_\textup{\'e}(n,K)\), where \(A\) is restricted to \'etale algebras.
	%\inote{ (Update mail from Milan: should we change k to a field?)}
\end{definition}

\begin{comment}
For fixed \(n\) and \(K\) there are only finitely many such \(A\), so the computation of \(M(n,K)\) is a finite problem.
The quantity \(M(n,\F_p)\) provides an upper bound on clique sizes.
\end{comment}

We will focus particularly on algebras over finite fields. In this case there are only finitely many isomorphism classes of commutative $K$-algebras of dimension $n$, so the computation of \(M(n,K)\) is a finite problem.

\begin{proposition}\label{prop:LleqM}
	If \(\mathcal{O}\) is an order of rank \(n\) and \(p\) is prime, then \(\omega(\mathcal{O})\leq M(n,\F_p)\).
\end{proposition}
\begin{proof}
	Every such \(\mathcal{O}\) gives rise to an \(\F_p\)-algebra \(A=\mathcal{O}/p\mathcal{O}\) of rank \(n\), and each exceptional clique \(\clq\subseteq \mathcal{O}\) satisfies \(N_{A/\F_p}(x-y)\equiv N_{\mathcal{O}/\Z}(x-y) = \pm 1\ (\textup{mod }p)\) for all distinct \(x,y\in \clq\).
\end{proof}

Computation of a "clique" in \(\F_{p^n}^*\) could be made more efficient by the following observation: the graph on \(\F_{p^n}^*\) with edges \(\{\{x,y\} \mid N(x-y)=\pm1\}\) is \emph{circulant} because the cyclic group \(\F_{p^n}^*\) acts on the vertices by multiplication. However, without using more structure, finding a maximum clique is still hard (see \cite{CODENOTTI1998123}).
Instead, we simplify computing an upper bound with the following two results.

\begin{proposition}
	For all fields \(K\) and \(n\in\Z_{\geq 0}\) there exists an \'etale commutative \(K\)-algebra \(B\) of rank \(n\) such that \(M(n,K)=M(B/K)\).
\end{proposition}
\begin{proof}
	It suffices to show that for every commutative \(K\)-algebra \(A\) there exists an \'etale commutative \(K\)-algebra \(B\) of the same rank together with a norm-preserving \(K\)-linear map \(A\to B\). 
	Using how norms interact with products and towers of algebras, we may assume without loss of generality that \(A\) is local with maximal ideal \(\m\) and residue field \(K\).
	By \cite[Proposition~III\S{}9.4.5]{Bourbaki} we may take the diagonal embedding \(A\to (A/\m)^n\).
\end{proof}

\begin{corollary}\label{cor:split_bound}
	If \(K\) is a field with algebraic closure \(L\), then \(M(n,K)\leq M(L^n/L)\).
\end{corollary}
\begin{proof}
	By the proposition, \(M(n,K)\) is attained by a clique in some \'etale \(A\). We have an embedding \(f:A\to A\tensor_K L \cong L^n\), and \(N_{A/K} = N_{A\tensor_K L /L} \circ f\).
\end{proof}

We first use this method to produce a general upper bound on the clique number.

\begin{lemma}\label{lem:ugly_matrix}
	Let $m\in\Z_{>0}$, let \(p\) be \(0\) or a prime greater than \(m^{m/2}\), and let $M\in\textup{Mat}_m(\Z/p\Z)$ have a zero diagonal and $\pm1$ coefficients elsewhere. Then $\rk(M)\ge m-1$.
\end{lemma}

\begin{proof}
	If \(p=2\), then \(M-1\) is the all-one matrix, which has rank \(1\). Using the rank inequality \(\rk(1)\leq\rk(M)+\rk(1-M)\) we get \(\rk(M)\geq m-1\).
	If \(p=0\), then the result follows from the previous case.
	Suppose now that \(p>m^{m/2}\) is prime and let \(N\) be the lift of \(M\) to \(\Z\) with coefficients in \(\{-1,0,1\}\). 
	By the previous case \(N\) has an \((m-1)\times(m-1)\) minor \(N'\) with non-zero determinant.
	Moreover, \(|\det(N')|\leq m^{m/2}<p\) by Hadamard's inequality, so the corresponding minor \(M'\) of \(M\) is invertible. Hence \(\rk(M)\geq m-1\).
\end{proof}

\begin{theorem}\label{thm:field_upper_bound}
	Let \(n\in\Z_{\geq0}\). Then for any field \(K\) of characteristic \(0\) or sufficiently large prime characteristic it holds that \(M(n,K)\leq 2^n+1\).
\end{theorem}
\begin{proof}
	By Corollary~\ref{cor:split_bound} we may assume \(K\) is algebraically closed and consider \(A=K^n\).
	Suppose $\clq\subseteq A$ is such that $N_{A/K}(x-y)=\pm1$ for all distinct $x,y\in \clq$.
	Let \(V\subseteq K[t_1,\dotsc,t_n]\) be the \(K\)-vector space spanned by the \(2^n\) monomials of degree \(0\) or \(1\) in each variable. 
	Consider the map \(\clq\to V\) given by \((x_1,\dotsc,x_n)\mapsto \prod_i (x_i-t_i)\) and let \(f:K^\clq\to V\) be its \(K\)-linear extension. 
	Conversely, consider the evaluation map \(g:V\to K^\clq\) given by \(P\mapsto (P(y))_{y\in \clq}\).
	Then \(\#\clq-1\leq \rk(g\circ f) \leq \rk(f) \leq \dim_K V = 2^n\), where the first inequality is Lemma~\ref{lem:ugly_matrix}. It follows that \(M(n,K)-1\leq2^n\).
\end{proof}

\begin{remark}
	In small characteristic, we might get cliques of larger size, for example $M(\F_{3^n}/\F_3)=3^n$ for all \(n\geq0\). More generally, if $p$ is an odd prime and \(n\in\Z_{\geq 1}\) is a multiple of \(d=(p-1)/2\), then \(\F_{p^{n/d}}\subseteq \F_{p^n}\) is a clique and \(M(n,\F_p)\geq p^{n/d}\).
\end{remark}

\begin{corollary}
	If \(L\) is a number field of degree \(n\) that is Galois over \(\Q(\sqrt{D})\) for some square-free \(D<-3\), then \(\omega(\mathcal{O}_L)\leq 2^{n/2}+1\). In particular this holds when \(L/\Q\) is Galois and \(\Delta_{L/\Q}<-3\). For \(D=-3\), we have \(\omega(\mathcal{O}_L)\leq 3^{n/2}\).
\end{corollary}
\begin{proof}
	Let \(\clq\subseteq\mathcal{O}_L\) be an exceptional clique. 
	Then \(\#\clq\leq M(n/2,\Q(\sqrt{D}))\leq 2^{n/2}+1\) by combining Theorem~\ref{thm:field_upper_bound} and the fact that \(N_{L/\Q(\sqrt{D})}(\mathcal{O}_L^*)=\{\pm1\}\). 
	In the case that \(L/\Q\) is Galois and \(\Delta_{L/\Q}<-3\), we have \(\smash{\Q(\sqrt{\Delta_{L/\Q}})}\subseteq L\). 
	In the case that \(D=-3\), the prime \(3\) is not inert in \(\mathcal{O}_L\), and thus \(\mathcal{O}_L\) has a quotient of size at most \(3^{n/2}\).
\end{proof}

Theorem~\ref{thm:field_upper_bound} provides no direct improvement over Corollary~\ref{cor:clique_bound}.ii on the clique number of orders, although it is interesting that we obtain a similar bound for number fields. However, the theorem, together with Proposition~\ref{prop:LleqM}, implies that every sufficiently large prime can be an obstruction to reaching a clique size of \(2^n\). This motivates us to apply this method to small rank orders.

\begin{theorem}\label{thm:computational_commutative}
	Exceptional cliques in orders of rank \(4\) and \(5\) have respectively at most \(10\) and \(11\) elements, and these bounds are sharp.
\end{theorem}
\begin{proof}
	We have computed \(M(4,\F_7)=10\) and \(M(5,\F_{19})=11\).
	For \(n=4\), this bound is attained by the clique of Example~\ref{ex:nf_deg_4}, and for \(n=5\) by \(\Z[\zeta_{11}+\zeta_{11}^{-1}]\).
\end{proof}

A special case of the function \(M\) has been well studied in the context of generalized Paley graphs (see \cite{PaleyGraphs}).
Let $q$ be a prime power and $d$ a divisor of $q-1$ such that $(q-1)/d$ is even if $q$ is odd. 
The generalized Paley graph $P(q,d)$ is the graph with vertex set $\F_q$ and edge set $\{\{x,y\}\,|\,x-y\in(\F_q^*)^d\}$, where $(\F_q^*)^d$ is the set of $d$-th power elements in $\F_q^*$.

The best known upper bound on the clique number of $P(q,d)$ in the case $d\mid p-1$ is $O(\sqrt{q/d})$.
But $M(\F_q/\F_p)$ is simply the clique number of $P(q,(p-1)/2)$, so Theorem \ref{thm:field_upper_bound} is an improvement on this bound for large \(p\) and \(d=(p-1)/2\), because our bound is independent of $p$.  
Using the same ideas, we extend our approach to a larger class of Paley graphs to obtain two new upper bounds.

Let $s_p(a)$ denote the sum of the digits of the integer $a$ written in base $p$, and let \(\omega(G)\) denote the clique number of a graph \(G\).

\begin{proposition}
	For \(p\) prime, \(q>1\) a power of \(p\), and \(d\mid q-1\) we have $\omega(P(q,d))\le 2^{s_p(\frac{q-1}{d})}+1$.
\end{proposition}

\begin{proof}
	Write \(q=p^n\) and \((q-1)/d=\sum_i a_i p^i\) with \(0\leq a_i < p\).
	Consider the linear map \(f:\F_q\to A=\F_q^{a_0+a_1+\dotsm}\) given by \(x\mapsto (x,\dotsc,x,x^p,\dotsc,x^p,\dotsc)\), where \(\dim_{\F_q}(A)=s_p((q-1)/d)\).
	Let \(C\) be the image under \(f\) of a clique of \(P(q,d)\).
	Then \(N_{A/\F_q}(x-y)=1\) for all distinct \(x,y\in C\). The matrix \(N=(N_{A/\F_q}(x-y))_{x,y\in C}\) has rank at least \(\#C-1\), since \(N+1\) is the all-one matrix. 
	The result now follows analogously to the proof of Theorem~\ref{thm:field_upper_bound}.
\end{proof}

This bound has interest only if $s_p(\frac{q-1}{d})\le \frac{1}{2}\log_2(q)$, otherwise it is worse than the trivial bound $\omega(P(q,d))\le\sqrt{q}$. Now we give another bound in the case $d\mid p-1$ and $p$ is large. This condition on \(d\) is `complementary' to the condition \(d\mid(q-1)/(p-1)\) studied by \cite{PaleyGraphs}.

\begin{lemma}\label{lem:ugly_matrix_cyclotomic}
	Let $m,r\in\Z_{>0}$, let \(p\) be a sufficiently large prime, and let $M\in\textup{Mat}_m(\overline{\F}_p)$ have a zero diagonal and $r$-th root of unity coefficients elsewhere. Then $\rk(M)\ge \sqrt{m-1}$.
\end{lemma}

\begin{proof}
	As in Lemma~\ref{lem:ugly_matrix} it suffices to consider the characteristic \(0\) case of such matrices over \(\Q(\zeta_r)\) instead of \(\overline{\F}_p\).
	Write \(\overline{M}\) for the matrix obtained from \(M\) by applying complex conjugation coefficient-wise.
	Consider the minor \(N\) of \(M\tensor\overline{M}\) corresponding to the basis \(\{e_i\tensor e_i\,|\, 1\leq i\leq m\}\), and note that \(N+1\) is the all-one matrix. Hence \(m-1\leq \rk(N)\leq\rk(M\tensor \overline{M})=\rk(M)\cdot\rk(\overline{M})=\rk(M)^2\).
\end{proof}

\begin{proposition}
	Fix $n\in\Z_{>0}$ and $c\in\R_{>0}$. For all $p$ sufficiently large and \(d\mid p-1\) such that $(p-1)/d\le c$, it holds that $\omega(P(p^n,d))\le 4^n+1$.
\end{proposition}

\begin{proof}
	Let $q=p^n$ and consider the map $\F_q\to A=\F_q^{n}$ sending $x$ to $\smash{(x,x^p,...,x^{p^{n-1}})}$. 
	The image of a clique in $P(q,d)$ under this map is a set $C\subseteq A$ with $N_{A/\F_q}(x-y)^{(p-1)/d}=1$ for all distinct $x,y\in C$. 
	We apply again the same ideas as in the proof of Theorem \ref{thm:field_upper_bound}: the matrix \(M=(N_{A/\F_q}(x-y))_{x,y\in C}\) satisfies the conditions to Lemma~\ref{lem:ugly_matrix_cyclotomic} ($(p-1)/d\le c$ allows one to fix $r$ in the lemma), so \(\#C-1\leq\rk(M)^2\leq (2^n)^2\).
\end{proof}

\section{Spreads}\label{sec:spreads}

In this section we provide a more geometric interpretation to cliques. 
It allows us to give an easier proof of the result by Leutbecher and Niklasch given in Theorem~\ref{thm:nt_solution}.
In this section \(\nrg\) is a ring, \(M\) is a non-zero (left) \(\nrg\)-module, and \(E=\End_\nrg(M)\).
Note that \(\nrg\) need not be commutative.

\begin{definition}[Spread]
	A \emph{spread} of \(M\) is an \(\nrg\)-module \(N\) together with a set \(S\) of \(\nrg\)-submodules of $N$ with \(U\cong M\) for each \(U\in S\) and \(U\oplus V = N\) for each pair of distinct \(U,V\in S\). 
	%A \emph{pointed spread} of \(M\) is a spread containing \(V_\infty := 0 \times M\).
	A \emph{pointed spread} of \(M\) is a set \(S\) of \(\nrg\)-submodules of $M^2$ with \(U\oplus V = M^2\) for each pair of distinct \(U,V\in S\), and with \(V_\infty:= 0\times M \in S\).
\end{definition}

Each pointed spread \(S\) of \(M\) is a spread, because for each \(V\in S\setminus \{V_\infty\}\) it holds that \(V \cong (V \oplus V_\infty)/V_\infty \cong M^2/V_\infty \cong M\).
Conversely, by choosing an isomorphism \(N\cong M^2\) and considering the action of \(\Aut(M^2)\), each spread containing at least two elements is under this action equivalent to a pointed spread.

\begin{theorem}\label{thm:clique_spread}
	Let \(n\in\Z_{\geq0}\).
	There is a bijection between the exceptional cliques of \(E\) of size \(n\) and the pointed spreads of \(M\) of size \(n+1\).
\end{theorem}
\begin{proof}
	First we construct the map from cliques to pointed spreads.
	Suppose \(\clq\subseteq E\) is a clique of size \(n\).
	For \(a\in \clq\) consider the \(k\)-linear map \(M\to M^2\) given by \(\left(\begin{smallmatrix}1 \\ a\end{smallmatrix}\right)\) and let \(V_a\) be its image, the graph of \(a\).
	We will show that \(S=\{V_a\,|\, a\in \clq\}\cup\{V_\infty\}\) is a spread.
	Let \(a,b\in \clq\) be distinct. 
	To show that \(V_a \oplus V_b = M^2\), it suffices to show that the \(k\)-linear map \(M^2\to M^2\)  given by \(\big(\begin{smallmatrix} 1 & 1 \\ a & b \end{smallmatrix}\big)\) is an isomorphism.
	By subtracting the first column from the second we obtain \(\big(\begin{smallmatrix} 1 & 0 \\ a & b-a \end{smallmatrix}\big)\), which is invertible because \(b-a\) is invertible. For \(V_a\oplus V_\infty\) we need to consider the map \(\big(\begin{smallmatrix} 1 & 0 \\ a & 1 \end{smallmatrix}\big)\), which is also clearly invertible.
	Finally, we remark that \(\#S=n+1\).
	
	Suppose \(S\) is a pointed spread of \(M\) of size \(n+1\).
	Since each \(V\in S\) is isomorphic to \(M\) there exist \(x_V,y_V\in E\) such that \(V\) is the image of \(\left(\begin{smallmatrix}x_V \\ y_V\end{smallmatrix}\right)\).
	Let \(S'=S\setminus\{V_\infty\}\).
	Because \(V\oplus V_\infty = M^2\) for each \(V\in S'\), the matrix \(\left( \begin{smallmatrix} x_V & 0 \\ y_V & 1 \end{smallmatrix} \right)\) must be invertible and hence \(x_V\in E^*\).
	It follows that \(V\) is the image of \(\left( \begin{smallmatrix} 1 \\ a_V \end{smallmatrix} \right)\), where \(a_V=y_V \cdot x_V^{-1}\).
	As \(U\oplus V = M^2\) for each pair of distinct \(U,V\in S'\), we similarly obtain \(a_U-a_V\in E^*\).
	Hence \(\clq=\{a_V\,|\, V\in S'\}\) is a clique, and it has size \(n\).
\end{proof}

Now we restrict our attention to a simple kind of spreads: we assume $\nrg$ is the ring of integers of some number field and take $M=\nrg$. In this case $E\cong\nrg$, so Theorem \ref{thm:clique_spread} produces cliques in $\nrg$. Proposition \ref{prop:no_galois_clique} shows that non-trivial Galois-invariant cliques do not exist. On the contrary, Galois-invariant spreads do, and they allow us to give a new interpretation of the best known infinite family of commutative cliques.

\begin{proposition}\label{prop:spread_example}
	If $\nrg$ is the ring $\Z[\zeta+\zeta^{-1}]$, where \(\zeta\) is a primitive root of unity of prime order \(p\), 
	then $\{\nrg \zeta^i\,|\, 0\le i < p\}$ is a spread of \(R\) in $\Z[\zeta]$. 
\end{proposition}

%\begin{proof}
%The ring of integers of the real subfield of the $p$-th cyclotomic field is exactly $\Z[\zeta+\zeta^{-1}]$ and one easily checks that the conditions of Theorem \ref{thm:galois_spread} are met.
%\end{proof}

\begin{proof}
	From $\zeta^2=\zeta(\zeta+\zeta^{-1})-1$ it is easy to see that $\nrg\oplus  \nrg \zeta =\nrg[\zeta]=\Z[\zeta]$. 
	By Galois action, we in fact get $\nrg\oplus \nrg\zeta^i=\Z[\zeta]$ for any $1\le i < p$. 
	Finally, $\nrg \zeta^i+\nrg\zeta^j=\zeta^i(\nrg+\zeta^{j-i}\nrg)=\zeta^i\Z[\zeta]=\Z[\zeta]$. 
	Thus we indeed have a spread.
\end{proof}

This spread together with Theorem~\ref{thm:clique_spread} gives us an alternative proof of Theorem~\ref{thm:nt_solution} when we decompose \(\Z[\zeta]=R\zeta \oplus R \cong R^2\), producing the exact same clique.
The construction from Proposition~\ref{prop:spread_example} gives rise to another Galois-invariant spread.

\begin{example}
	Consider $R=\Z[\eta_5,\eta_7]$ and $\alpha=\eta_{35}$, where $\eta_n=\zeta_n+\zeta_n^{-1}$ and \(\zeta_n\) is a primitive \(n\)-th root of unity. Then \(\{R\}\cup \{R\sigma(\alpha)\,|\, \sigma \in \textup{Gal}(\Q(\eta_{35})/\Q)\}\) is a spread of size \(13\), although this is less obvious. Consequently, we obtain a clique of size \(12\) in \(R\) of rank \(6\). 
\end{example}

Finally, we use spreads to turn commutative cliques into larger cliques in larger dimension, at the cost of commutativity.

\begin{proposition}\label{prop:spread_projections}
	Let \(S\) be a spread of \(M\) of size at least \(2\). Then there exists a way to map every $V\in S$ to some $f_V\in\End_E(M^2)$ such that \(\im(f_V)=\ker(f_V)=V\) for all \(V\in S\), and the set $\{f_V\}_{V\in S}$ is an exceptional clique.
\end{proposition}
\begin{proof}
	Choose a permutation \(g:S\to S\) without fixed points.
	For \(V\in S\) we have \(M^2\cong V\oplus g(V)\) and \(V\cong g(V)\), so we may indeed choose \(f_V\in\End_E(M^2)\) so that \(f_V(V)=0\) and \(f_V(g(V))=V\).
	
	Now let \(f_\cdot\) be any such map. For distinct \(U,V\in S\), the map \(f_U-f_V\) acts as the direct sum of two restrictions \(f_U:V\to U\) and \(-f_V:U\to V\), so it suffices by symmetry to show that \(f_U:V\to U\) is an isomorphism. However, this follows because the natural map \(V\to M/\ker(f_U)\) is an isomorphism.
\end{proof}

\begin{theorem}\label{thm:clique_doubling}
	Let \(\clq\subseteq E\) be a commutative exceptional clique. Then \(\textup{Mat}_2(E)\) has an exceptional clique of size \(2\cdot \#\clq\). Consequently \(\omega(\textup{Mat}_2(E))\geq 2\omega_{\textup{com}}(E)\).
\end{theorem}
\begin{proof}
	We may assume \(0\in \clq\). By replacing \(E\) by the ring generated by \(\clq\) we may assume \(E\) is commutative. Combining Theorem~\ref{thm:clique_spread} and Proposition~\ref{prop:spread_projections}, we obtain a clique \(\clqb\subseteq\End_E(E^2)\) of nilpotent elements and of size \(\#\clq+1\). We claim that \(\clqb\) together with the image of \(\clq\setminus\{0\}\) in \(\End_E(E^2)\), which is central, is an exceptional clique. For this it suffices to note that \(x-f\) for \(x\in \clq\setminus \{0\}\) and \(f\in \clq'\) is a unit, which follows from the fact that \((x-f)(x+f)=x^2-f^2=x^2\in E^*\).
\end{proof}

\section{Cliques from cyclic algebras}\label{sec:cyclicalgebras}

In this section we will use a non-commutative arithmetic construction to produce exceptional cliques of quadratic rate. 
%The following proposition shows that in the commutative setting the Galois action does not interact too nicely with exceptional cliques.
Contrary to the commutative case (see Proposition~\ref{prop:no_galois_clique}), in the non-commutative setting we are able to make better use of the Galois action.
In this section \(L/K\) is a cyclic Galois extension of number fields, and \(G=\langle \sigma\rangle\) is the Galois group.

\begin{definition}[Cyclic algebra]
	Consider the free left \(L\)-module \(L^{(G)}\) which we equip with a \(K\)-bilinear multiplication \((a\sigma^i) \cdot (b \sigma^j) = (a\sigma(b)) (\sigma^{i+j})\) for all \(a,b\in L\) and \(i,j\in\Z\). 
	We write \(\mathcal{A}_{L/K}\) for the subring \((\mathcal{O}_L)^{(G)}\), which is an \(\mathcal{O}_K\)-algebra.
\end{definition}

Note that \(\mathcal{O}_L\) is a non-zero \(\mathcal{A}_{L/K}\)-module.
Hence we have a ring homomorphism \\ \(\mathcal{A}_{L/K}\to\End(\mathcal{O}_L)\cong \textup{Mat}_{[L:\Q]}(\Z)\). It now suffices to construct a large clique in \(\mathcal{A}_{L/K}\).

\begin{lemma}\label{lem:unit_norm}
	If \(N_{L/K}(u)\) is an exceptional unit for \(u\in \mathcal{O}_L\), then \(u\sigma\in \mathcal{A}_{L/K}\) is an exceptional unit.
\end{lemma}
\begin{proof}
	Write $n=[L:K]$. By assumption we have units
	\[ (u\sigma)^n = \big(u\sigma(u)\dotsm \sigma^{n-1}(u)\big) \sigma^n= N_{L/K}(u) \quad\textup{and}\quad (1-u\sigma)\cdot \sum_{i=0}^{n-1} (u\sigma)^i = 1 - (u\sigma)^n=1-N_{L/K}(u),\]
	from which it follows respectively that \(u\sigma\) and \(1-u\sigma\) are invertible.
\end{proof}

\begin{proposition}\label{prop:nt_square_clique}
	Suppose that \(\clq\subseteq \mathcal{O}_L\) is a set such that \(N_{L/K}(x-y)=\pm1\) for all distinct \(x,y\in \clq\), and that \(N_{L/K}(u^2)\) is an exceptional unit for some \(u\in \mathcal{O}_L\). Then \(\clqb=\{x+yu\sigma\,|\, x,y\in \clq\}\subseteq\mathcal{A}_{L/K}\) is an exceptional clique.
\end{proposition}
\begin{proof}
	Note that in particular \(\clq\) is an exceptional clique.
	Let \(z_1,z_2\in \clqb\) be distinct and write \(z_i=x_i+y_iu\sigma\). If \(y_1=y_2\), then \(z_1-z_2=x_1-x_2\in \mathcal{O}_L^*\), and similarly \(z_1-z_2\) is a unit if \(x_1=x_2\). Otherwise, we have \(z_1-z_2=a(1-bu\sigma)\) for some \(a,b\in\mathcal{O}_L^*\) with \(N_{L/K}(b)=\pm1\).
	By assumption on \(u\) both \(\pm N_{L/K}(u)\) are exceptional units, and thus so is \(N_{L/K}(bu)\).
	Then \(z_1-z_2\) is a unit by Lemma~\ref{lem:unit_norm} and \(\clqb\) is an exceptional clique.
\end{proof}

\begin{lemma}\label{lem:exist_square_exceptional}
	Suppose \(p,q\) are distinct odd primes and let \(K=\Q(\eta_q)\), where \(\eta_k=\zeta_k+\zeta_k^{-1}\).
	Then there is some \(u\in\mathcal{O}_{K(\eta_p)}\) for which \(N_{K(\eta_p)/K}(u^2)\) is an exceptional unit if and only if \(p\not\equiv \pm 1\ (\textup{mod }q)\).
\end{lemma}
\begin{proof}
	Consider \(u=\eta_p-\eta_q\) and
	\[v=N_{K(\zeta_p)/K}(u)=N_{K(\zeta_p)/K}\big(\zeta_p^{-1}\cdot (\zeta_p-\zeta_q) \cdot (\zeta_p-\zeta_q^{-1})\big) = 1\cdot \frac{\zeta_q^p-1}{\zeta_q-1}\cdot\frac{\zeta_q^{-p}-1}{\zeta_q^{-1}-1}.\]
	Write $a$ and $b$ for (respectively) the numerator and denominator of the fraction on the right. Then \(v\) is a unit, since \(a\) and \(b\) are conjugate. 
	Moreover, since \(u\in K(\eta_p)\), we have \(N_{K(\eta_p)/K}(u^2)=v\). If \(p\not\equiv \pm 1\ (\textup{mod }q)\), then
	\[ 1-v = \frac{b-a}{b} = \frac{\zeta_q + \zeta_q^{-1} -\zeta_q^p-\zeta_q^{-p}}{b} = \zeta_q^{-p} \cdot \frac{(\zeta_q^{p+1}-1)(\zeta_q^{p-1}-1)}{b}, \]
	where the numerator and denominator on the right are again conjugate.
	Hence \(1-v\) is an exceptional unit.
	
	Assume conversely that $p\equiv\pm1\bmod q$. This implies that $p$ splits completely in $K$, so if $\mathfrak{p}$ is a prime of $K$ above $p$, then $\mathcal{O}_K/\mathfrak{p}\cong\F_p$. Now $\mathfrak{p}$ is totally ramified in $K(\eta_p)$, so there is a unique prime $\mathfrak{P}$ of $K(\eta_p)$ above $\mathfrak{p}$ and $\ent_{K(\eta)}/\mathfrak{P}\cong\F_p$. For $u\in\ent_{K(\eta)}$, we therefore have $N_{K(\eta)/K}(u)\equiv x^{\frac{p-1}{2}}\equiv \pm1 \bmod \mathfrak{P}$ so $N_{K(\eta_p)/K}(u^2)-1$ is never a unit.
\end{proof}

\begin{theorem}\label{th:main}
	Let \(p,q\) be distinct odd primes such that \(p\not\equiv \pm 1\ (\textup{mod }q)\).
	Then for $m=(q-1)(p-1)/4$, $\omega(\Mat_{m}(\Z))\geq (p-1)^2$
\end{theorem}
\begin{proof}
	Take \(K=\Q(\eta_q)\) and \(L=K(\eta_p)\), and note that \([L:\Q]=(p-1)(q-1)/4\). 
	By Theorem~\ref{thm:nt_solution}, we have a clique \(\clq\) of size \(p-1\) in \(\mathcal{O}_L\).
	For each pair of distinct \(u,v\in \clq\) we have \(N_{L/K}(u-v)=N_{\Q(\eta_p)/\Q}(u-v)=\pm 1\).
	By Lemma~\ref{lem:exist_square_exceptional} we may apply Proposition~\ref{prop:nt_square_clique} to obtain an exceptional clique \(\clqb\subseteq\mathcal{A}_{L/K}\) of size \((p-1)^2\).
	Lastly, we have a ring homomorphism \(\mathcal{A}_{L/K}\to \Mat_{[L:\Q]}(\Z)\) that preserves the clique by Proposition~\ref{prop:injective_quotient}.
\end{proof}

Taking $q=5$ produces the best currently known infinite family of exceptional cliques.

\begin{corollary}\label{cor:q=5}
	If $p\not\equiv \pm1\ (\textup{mod }5)$ is prime, then \(\omega(\Mat_{p-1}(\Z))\geq(p-1)^2\) \qed
\end{corollary}

\begin{remark}\label{rem:goldbach}
It is not unreasonable to expect that we may, analogously to Theorem~\ref{thm:goldbach}, extend the result of Corollary~\ref{cor:q=5} to all even dimensions \(n\) by writing \(n\) as a sum of \(p_i-1\), where the \(p_i\) are primes congruent to \(2\) or \(3\) modulo \(5\) and are all approximately equally large. 
If we define \(\rho(n)\) to be the maximum of \(\min_i p_i\) running over all decompositions \(n=\sum_i (p_i-1)\) with \(p_i\not\equiv \pm1\ (\textup{mod }5)\) prime, then a heuristic argument suggests that \(|\rho(n)-n/l(n)|=O(\log(n)^2)\), where \(l(n)\) is the minimum number of terms needed in the decomposition. For \(n\leq 10^8\), we verified experimentally that \(|\rho(n)-n/l(n)|\leq 30 \log(n)^2\) and that if \(l(n)\neq 1\), then \(l(n)=2\) if \(n\equiv 2,3,4\ (\textup{mod }5)\) and \(l(n)=3\) otherwise.
Therefore, if $n$ is even one can in practice construct an exceptional clique in $\Mat_n(\Z)$ of size around $(n/l(n))^2$.
\end{remark}

% Analogously to Theorem~\ref{thm:goldbach}, we can combine these cliques to get cliques in dimensions that are sums of $p_i-1$ where the $p_i$'s are primes congruent to $2,3\bmod 5$. To do this efficiently, we need to know if a given integer can decompose as a sum of these terms efficiently; the clique size is limited by the $\min$ of these terms.\\

% \textbf{Experimental facts:}\\
% Denote by $\mathbb{P}$ the set of primes that are congruent to 2 or 3 modulo 5. Let $n>0$ be an even integer such that $n+1\notin\mathbb{P}$.
% \begin{enumerate}
	% \item If $n=2,3,4\bmod5$, then $n=(p_1-1)+(p_2-1)$ whith $p_1,p_2\in\mathbb{P}$.
	% \item If $n=0,1\bmod5$, then $n=(p_1-1)+(p_2-1)+(p_3-1)$ whith $p_1,p_2,p_3\in\mathbb{P}$.
	% \end{enumerate}
% Denote by $\rho(n)$ the maximum of $\min_i (p_i-1)$ over all decompositions $n=\sum (p_i-1)$ with $p_i\in\mathbb{P}$. When $n$ is large, we expect $\rho(n)\sim \frac{n}{2}$ in the first case and $\rho(n)\sim\frac{n}{3}$ in the second. More precisely, if $l(n)=\begin{cases}2,\text{ if }n=2,3,4\bmod 5\\3,\text{ if }n=0,1\bmod 5\end{cases}$ a probabilistic argument suggests $|\frac{n}{l(n)}-\rho(n)|=O(\log^2(n))$.
% We computer-checked this for $n$ up to $10^8$ and found that the constant in the $O(\log^2(n))$ is $0.90$ in the average case and $30$ in the worst case. This could be refined by computing this constants for each class $\bmod 5$.\\
% Therefore, if $n$ is even one can in practice construct an exceptional clique in $\Mat_n(\Z)$ of size $\frac{n^2}{l(n)^2}$.

The cyclic algebra construction of cliques can, with a twist, also be used to construct some of the best known cliques in small dimensions. For the remainder of this section we fix some non-zero \(\eta\in\ent_L\).

\begin{definition}[Twisted cyclic algebra]
	Define the \emph{twisted cyclic algebra} $\A_{L/K}^\eta$ to be the algebra similar to $\A_{L/K}$ but with the relation $\sigma^{[L:K]}=\eta$ instead of $\sigma^{[L:K]}=1$.
\end{definition}

Note that \(\ent_L\) is generally not an \(\A_{L/K}^\eta\)-module, and thus we no longer have a morphism \(\mathcal{A}_{L/K}^\eta\to\Mat_{[L:\Q]}(\Z)\).
Instead, we obtain the following.

\begin{lemma}\label{lem:mat_embedding_twist}
	Let $F$ be an intermediate extension of $L/K$ and suppose that $\eta\in N_{L/F}(\ent_L)$. 
	Then there is a ring homomorphism $\A_{L/K}^\eta\to \Mat_{[F:K][L:\Q]}(\Z)$.
\end{lemma}

\begin{proof}
	Let \(x\in\mathcal{O}_L\) be such that \(N_{L/F}(x)=\eta\).
	Firstly, \(\mathcal{A}_{L/K}^\eta\) is a free left-\(\mathcal{A}_{L/F}^\eta\)-module with basis \(1,\sigma,\dotsc,\sigma^{[F:K]-1}\), and thus we have a morphism \(\mathcal{A}_{L/K}^\eta\to\textup{Mat}_{[F:K]}(\mathcal{A}_{L/F}^\eta)\).
	Secondly, the \(F\)-linear map \(\mathcal{A}_{L/F}^\eta\to \mathcal{A}_{L/F}\) given by \(a\tau^i \mapsto ax^i\tau^i\) is a ring homomorphism, where \(a\in \mathcal{O}_L\) and \(\tau=\sigma^{[F:K]}\) generates the Galois group of \(L/F\). 
	Thirdly, we have a morphism \(\mathcal{A}_{L/F}\to\textup{Mat}_{[L:\Q]}(\Z)\).
	These maps combined give the required ring homomorphism.
\end{proof}

We obtain the following twisted version of Proposition~\ref{prop:nt_square_clique}.

%\begin{lemma}\label{lem:unit_norm_twisted}
%If \(N_{L/K}(u)\eta\) is an exceptional unit for \(u\in \mathcal{O}_L\), then \(u\sigma\in \mathcal{A}_{L/K}\) is an exceptional unit.
%\end{lemma}

\begin{proposition}\label{prop:nt_square_clique_twisted}
	Suppose \(\clq\subseteq \mathcal{O}_L\) is a set such that \(N_{L/K}(x-y)=\pm1\) for all distinct \(x,y\in \clq\), and \(\eta^2\) is an exceptional unit, then \(\clqb=\{x+y\sigma\,|\, x,y\in \clq\}\subseteq\mathcal{A}_{L/K}^\eta\) is an exceptional clique. \qed
\end{proposition}

\begin{proof}
One can adapt Lemma \ref{lem:unit_norm} to the twisted setting and see that $\pm\sigma$ are exceptional units. The proof then follows that of Proposition \ref{prop:nt_square_clique}.
\end{proof}

\begin{example}\label{ex:n4c16}
	Consider \(K=\Q\) and \(F=L=\Q(\phi)\) with \(\eta=N_{L/F}(\phi)=\phi\).
	Since \(\Z[\phi]\) has a clique \(\clq\) of size \(4\), and \(\eta^2\) is an exceptional unit, we obtain by Proposition~\ref{prop:nt_square_clique_twisted} and Lemma~\ref{lem:mat_embedding_twist} a clique of size \(16=2^4\) in \(\textup{Mat}_{4}(\Z)\), which achieves the upper bound from Corollary~\ref{cor:clique_bound}. 
	As a clique in \(\textup{Mat}_2(\Z[\phi])\) it equals
	\[\Big\{\begin{pmatrix} x & y\phi \\ \sigma(y) & \sigma(x) \end{pmatrix} \Big|\, x,y\in \clq\Big\}.\]
\end{example}

\begin{example}\label{ex:n12c324}
	Let \(K=\Q\), \(F=\Q(\phi)\) and \(L=F(\eta_7)\). 
	The field \(L\) contains an exceptional clique of size \(18\).
	With \(\eta=\phi\) we obtain a clique of size \(18^2=324\) in dimension \(12\).
\end{example}

% \begin{proposition}\label{prop:square4}
	% Let \(p\geq 5\) be prime and \(n=(p-1)/2\). Then \(\textup{Mat}_{n^2}(\Z)\) contains an exceptional clique of size \(4n^2\).
	% \end{proposition}
% \begin{proof}
	% Let \(\eta=\zeta_p+\zeta_p^{-1}\), let \(G=\{\sigma_1,\dotsc,\sigma_n\}\) be the Galois group of \(\Q(\eta)\), and consider for \(x\in\Z[\eta]\) the \(\Z[\eta]\)-valued matrices
	% \[ 
	%   A_x = \begin{pmatrix} \sigma_1(x) & \\ & \ddots \\ & & \sigma_n(x)  \end{pmatrix} \quad\textup{and}\quad 
	%   B_x = \begin{pmatrix} 0 & & & \eta \sigma_n(x)  \\ \sigma_1(x) & \ddots & \\ & \ddots & 0 &  \\  & & \sigma_{n-1}(x) & 0 \end{pmatrix}.
	% \]
	% For all \(w,x,y,z\in\Z[\eta]\) it follows that
	% \[\det((A_w+B_x)-(A_y+B_z)) = \det(A_w-A_y)\pm\det(B_x-B_z) = N(w-y) \pm \eta N(x-z).\] 
	% Note that \(a+b\eta\) is a unit for all \(a,b\in\{-1,0,1\}\) unless \(a=b=0\).
	% By Theorem~\ref{thm:nt_solution} there is an exceptional clique \(\clq\subseteq\Z[\eta]\) of size \(2n\), so we may take \(\{ A_x + B_y \,|\, x,y\in \clq \}\subseteq\textup{Mat}_{n}(\Z[\eta])\hookrightarrow\textup{Mat}_{n^2}(\Z)\) as exceptional clique.
	% \end{proof}

\section{Cliques from coding theory}\label{sec:coding}

While Theorem~\ref{th:main} allows one to produce cliques in large dimension from cliques in small dimension, this construction cannot be applied inductively because the commutativity assumption is lost in the process, which is a problem Theorem~\ref{thm:clique_doubling} also suffers from.
In this section we describe an inductive exceptional clique construction using coding theory that does not depend on a commutativity assumption.

\begin{proposition}\label{prop:coding_inductive_step}
	Suppose \(\nrg\) is a ring and \(\clq\subseteq \nrg\) is an exceptional clique. 
	Suppose \(V\subseteq \clq^{2t}\) is a code where for any two distinct words \((v_1,\dotsc,v_{2t}),(w_1,\dotsc,w_{2t})\in V\) the set \(\{i \,|\, v_i=w_i\}\) is non-empty and consists of only odd or only even numbers.
	Then \(\textup{Mat}_{t}(\nrg)\) has an exceptional clique of size \(\#V\).
\end{proposition}
\begin{proof}
	Write \(s=2t\).
	To each \(w=(w_1,\dotsc,w_s)\in \clq^s\) we associate the matrix
	\[A_w=\begin{pmatrix} w_1 & & & & w_s\\ w_2 & w_3 \\ & w_4 & \smash{\ddots} \\ & & \smash{\ddots} & w_{s-3} \\ & & & w_{s-2} & w_{s-1} \end{pmatrix} \in \textup{Mat}_{t}(\nrg).\]
	The \(\nrg\)-linear isomorphism \(\sigma:\textup{Mat}_t(\nrg)\to\textup{Mat}_t(\nrg)\) obtained by multiplying on the right by the permutation matrix of \((1\ 2\ \dotsm\ t)\) and then taking the transpose preserves units and induces the map \((w_1,\dotsc,w_s)\mapsto(w_s,w_1,\dotsc,w_{s-1})\) on \(\clq^s\). 
	%Note that \(\sigma^s=\id\), and thus the group \(\Z/s\Z\) acts on both \(\textup{Mat}_n(\Z)\) and \(C\), and this action is transitive on the coordinates of \(C\). 
	
	Suppose \(v,w\in V\) are distinct words. 
	We will show that \(A_v-A_w\) is invertible.
	There exists some \(i\) such that \(v_i=w_i\), and by the action of \(\langle\sigma\rangle\) we may assume without loss of generality that \(i=s\).
	Then \(A_v-A_w\) is a lower triangular matrix, and by assumption on the code it has invertible diagonal entries, hence it is invertible even when \(\nrg\) is non-commutative.
	%If \(k\) is odd, we may extend the code to \(C'\subseteq \clq^{k+1}\) by sending \((w_1,\dotsc,w_k)\in C\) to \((w_1,\dotsc,w_k,w_{k-1})\) and apply the previous case.  
\end{proof}

We will apply this proposition to the following linear code.

\begin{example}\label{ex:RS_def}
	Let \(q\) be a prime power and write \(\mathbb{P}=\mathbb{P}^1(\F_q)\) for the set of \(1\)-dimensional \(\F_q\)-submodules of some fixed \(2\)-dimensional module \(V\).
	Consider the natural map \(V\to\bigoplus_{L\in\mathbb{P}} (V/L)\) and note that it is injective.
	Its image is the \emph{Reed--Solomon code} of dimension \(2\) and length \(q+1\).
	The images of any two distinct \(x,y\in V\) are the same at coordinate \(L\) if and only if \(L=\F_q\cdot(x-y)\).
	Hence any two distinct words are equal in exactly one place.
\end{example}

\begin{corollary}\label{cor:coding_inductive_step}
	If \(\nrg\) is a ring and \(\clq\subseteq \nrg\) is an exceptional clique of prime-power size \(q\), then for every integer \(t\geq (q+1)/2\) the ring \(\textup{Mat}_t(\nrg)\) has an exceptional clique of size \(q^2\).
\end{corollary}
\begin{proof}
	Following the example above, choose bijections \(\mathbb{P}\to\{1,\dotsc,q+1\}\) and \(V/L\to \clq\) for all \(L\in\mathbb{P}\), and let \(W\subseteq \clq^{q+1}\) be the corresponding Reed--Solomon code.
	We embed \(W\) in \(\clq^{2t}\) by \((w_1,\dotsc,w_{q+1})\mapsto (w_1,\dotsc,w_{q+1},w_{q+2}',\dotsc,w_{2t}')\), where \(w_i'=w_1\) or \(w_i'=w_2\) depending on whether \(i\) is odd or even respectively. The result follows from Proposition~\ref{prop:coding_inductive_step}.
\end{proof}

\begin{remark}
	The Reed--Solomon code is constructed using lines of the affine plane over $\F_q$. 
	Abstractions of the affine plane have been studied where the plane is just a set of points \(P\) together with a set of subsets \(L\) of \(P\), the lines, satisfying various axioms \cite{AffinePlanes}.
	One may for each class of parallel lines \([\ell]\) choose a bijection between \([\ell]\) and a clique and proceed as in Corollary~\ref{cor:coding_inductive_step}.
	However, it is suspected that all finite affine planes have prime-power order \cite[p.\ 144]{AffinePlanes}, so we obtain no relaxation of the conditions to Corollary~\ref{cor:coding_inductive_step}.
\end{remark}

By recursive application of Corollary~\ref{cor:coding_inductive_step} we may now prove the following theorem.

\begin{theorem}\label{thm:coding}
	Fix \(d\geq1\) and a prime power \(q\), and suppose \(\textup{Mat}_d(\Z)\) contains an exceptional clique of size \(q\).
	Then there exist infinitely many \(n\in\Z_{\geq 0}\) such that \(\textup{Mat}_n(\Z)\) contains an exceptional clique of size \(\delta_{d,q}\cdot n \log n + O(n\log\log n)\), where 
	\[\delta_{d,q} = \frac{q}{d\log q}\cdot \prod_{i=0}^\infty \Big(1+\frac{\varepsilon_q}{q^{2^i}}\Big)^{-1} \quad\textup{and}\quad \varepsilon_q = \begin{cases} 2 & \textup{if \(q\) is even}\\ 1 & \textup{if \(q\) is odd}\end{cases}.\]
\end{theorem}

Note that we may trivially choose \(d=1\) and \(q=2\) to obtain cliques with a super-linear rate.
As shown by Example~\ref{ex:n4c16}, we may also take \(d=4\) and \(q=16\) for a better constant factor \(\delta_{4,16}\approx 1.27\) compared to \(\delta_{1,2}\approx0.85\). 

\begin{proof}
	Let \(R_0=\textup{Mat}_{n_0}(\Z)\) for \(n_0=d\), which contains an exceptional clique of size \(s_0=q\).
	We may inductively construct a clique of size \(s_{k+1}=s_k^2\) in \(R_{k+1}=\textup{Mat}_{n_{k+1}}(\Z)\) for \(n_{k+1}=n_k\cdot (s_k+\varepsilon_q)/2\) using Corollary~\ref{cor:coding_inductive_step}.
	Here
	\[s_{k}=q^{2^{k}}\quad\textup{and}\quad n_{k}=d\cdot2^{-k}\cdot \prod_{i=0}^{k-1}\big(q^{2^i}+\varepsilon_q\big) = d \cdot 2^{-k}\cdot q^{2^k-1} \cdot \prod_{i=0}^{k-1} \Big(1+\frac{\varepsilon_q}{q^{2^i}}\Big). \]
	Note that
	\[\prod_{i=0}^{k-1} \Big(1+\frac{\varepsilon_q}{q^{2^i}}\Big)\leq \exp\Big( \sum_{i=0}^{k-1}\frac{\varepsilon_q}{q^{2^i}} \Big) \leq \exp\Big(\frac{\varepsilon_q\cdot q}{q-1}\Big),\]
	so the product converges as \(k\) tends to infinity.
	We similarly conclude that
	\begin{align*} 
		\prod_{i=k}^\infty \Big(1+\frac{\varepsilon_q}{q^{2^i}}\Big) 
		&\leq \exp\Big(\sum_{i=k}^\infty \frac{\varepsilon_q}{q^{2^i}} \Big) \leq \exp\Big( \frac{q}{q-1} \cdot \frac{\varepsilon_q}{q^{2^k}} \Big) \\ 
		&\leq 1 + \frac{q}{q-1} \cdot \frac{\varepsilon_q}{q^{2^k}} \cdot\exp\Big(\frac{q}{q-1} \cdot \frac{\varepsilon_q}{q^{2^k}}\Big) = 1+O\big(q^{-2^k}\big) = 1 + O(n_k^{-1}).
	\end{align*}
	We have
	\[2^k = \log_q n_k + O(k) = \log_q n_k + O(\log\log n_k),\]
	and similarly
	\begin{align*}
		s_k &=  n_k\cdot  2^k \cdot \delta_{d,q} \cdot (\log q) \cdot \prod_{i=k}^{\infty}\Big(1+\frac{\varepsilon_q}{q^{2^i}}\Big) \\ 
		&= \delta_{d,q} \cdot n_k \big(\log n_k + O(\log\log n_k) \big) \cdot \big(1+O(n_k^{-1})\big) \\
		&= \delta_{d,q} \cdot n_k \log n_k + O(n_k\log\log n_k),
	\end{align*}
	as was to be shown.
\end{proof}

One could imagine a stronger result can be obtained if we apply Proposition~\ref{prop:coding_inductive_step} to a better code. However, it turns out the Reed--Solomon code is optimal in the following sense.

\begin{proposition}
	For \(s\geq 1\) and \(\clq\) a finite set of cardinality \(\#\clq\geq 2\), let \(V\subseteq \clq^s\) be a code such that for any two distinct \(v,w\in V\) the set \(\{i\,|\, v_i=w_i\}\) is non-empty and consists of only odd or only even numbers. Then \(\# V\leq \# \clq^2\) and \(s \geq (\# V-1)/(\#\clq-1)\).
\end{proposition}
\begin{proof}
	If \(s = 1\), then \(\# V\leq 1\) and we are done. Suppose \(s>1\). 
	Let \(\pi_i:V\to \clq\) be the projection onto the \(i\)-th coordinate. If \(i+j \equiv 1\ (\textup{mod }2)\) and \(x\in \clq\), then the restriction of \(\pi_i\) to \(\pi_j^{-1}\{x\} \to \clq\) is injective.
	Hence \(\#\pi_i^{-1}\{x\}\leq \#\clq\) for every \(i\) and \(x\in \clq\). It follows that \(\# V= \#(\pi_1^{-1} \clq) \leq \#\clq^2\). Now let \(v\in V\). Then
	\[\#V-1=\# (V\setminus \{v\}) \leq \sum_{i} \# \big((\pi_i^{-1}\{v_i\} \big)\setminus\{v\}\big) \leq s(\# \clq-1),\]
	as was to be shown.
\end{proof}

\section{Cliques from graph theory}\label{sec:graph}

Our best two clique constructions, using cyclic algebras and using codes, can only produce large cliques in composite dimension. 
%In particular, we have obtained quadratic rate only for even dimension at best (see Remark~\ref{rem:goldbach}). 
Complementarily to this, in this section we will construct exceptional cliques in \(\textup{Mat}_n(\Z)\) using graphs when \(n\) is prime.

We first give some basic definitions from graph theory that are slightly modified for our purposes.

\begin{definition}
	A \emph{graph} is a symmetric \(\{0,1\}\)-valued square matrix with \(0\) diagonal. 
	A \emph{signed graph} is a symmetric \(\{-1,0,1\}\)-valued square matrix with \(0\) diagonal. 
	Let \(G\) be an \(n\times n\) signed graph.
	A \emph{vertex} of \(G\) is an integer in \(\{1,\dotsc,n\}\) and an \emph{edge} is a pair \(\{v,w\}\) of vertices such that \(G_{v,w}\neq 0\).
	For each vertex \(v\), the \emph{degree} of \(v\) in \(G\) is the number of non-zero entries in row \(v\) of \(G\). 
	The \emph{sign} of \(G\) is the product of all its non-zero entries.
	The \emph{Laplacian} \(\mathcal{L}(G)\) of \(G\) is \(G+D\), where \(D\) is the diagonal matrix such that all row-sums of \(G+D\) are \(0\).
	We say \(G\) is \emph{connected} if for every two vertices \(v,w\) of \(G\) there exists a sequence of vertices \((u_1,\dotsc,u_m)\) such that \((u_1,u_m)=(v,w)\) and \(G_{u_i,u_{i+1}}\neq 0\) for all \(1\leq i <m\).
	A \emph{subgraph} of \(G\) is an \(n\times n\) signed graph \(H\) such that for all edges \(\{v,w\}\) of \(H\) we have \(H_{v,w}=G_{v,w}\).
	We say \(G\) is a \emph{tree} if it is connected and has no connected subgraphs but \(G\).
	A \emph{spanning tree} of \(G\) is a subgraph which is a tree. 
	%Let \(G=(V,E)\) be a graph. A \emph{spanning tree} of \(G\) is a subgraph \(G'=(V,E')\) of \(G\) that is connected such that no subgraph of \(G'\) is connected.
	%The \emph{Laplacian} \(\mathcal{L}(G)\) of \(G\) is the symmetric matrix \((m_{v,w})_{v,w\in V}\) where 
	%\[ m_{v,w} = \begin{cases} 1 & \textup{if } v\neq w\textup{ and }\{v,w\}\in E \\ 0 & \textup{if }v\neq w\textup{ and }\{v,w\}\not\in E \\ -\#\{u \,|\, \{u,v\}\in E\} & \textup{if } v=w, \end{cases} \]
\end{definition}

The following theorem allows us to construct exceptional cliques in \(\textup{Mat}_n(\Z)\) from graphs.

\begin{theorem}[Matrix tree theorem \cite{MatrixTree}]\label{thm:matrix_tree}
	Let \(G\) be a signed graph. Then for each vertex \(v\) the value of \(\det(\mathcal{L}(G)^v)\), where \(-^{v}\) indicates the submatrix obtained by removing the \(v\)-th row and column, equals the number of spanning trees of \(G\) of positive sign minus the number of spanning trees of \(G\) of negative sign. \qed
\end{theorem}

\begin{definition}
	We will refer to a set \(\mathcal{G}\) of graphs on the same vertex set as a \emph{clique of graphs} if for every  vertex \(v\) the set \(\{ \mathcal{L}(G)^v \,|\, G\in\mathcal{G} \}\) is an exceptional clique.
\end{definition}
Note that cliques of graphs induce exceptional cliques consisting of symmetric matrices.
We will use the following graphs to construct cliques.

\begin{example}\label{ex:graphs}
	Let \(p\) be an odd prime. We define the following graphs on the vertex set \(\F_p\).
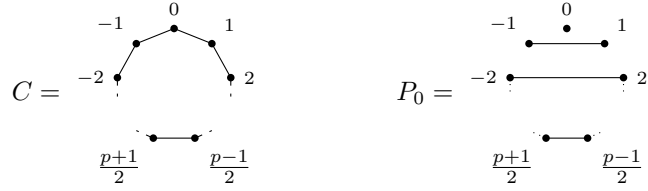
	\[C = \vcenter{\hbox{\begin{tikzpicture}[scale=.5,vrtx/.style args = {#1/#2}{%
					circle, draw, thick, fill=white,
					minimum size=.1, label=#1:#2}]
				\draw (-1.5,-1.3) -- (-1,-.4) -- (0,0) -- (1,-.4) -- (1.5,-1.3);
				\draw[dashed] (1.5,-1.3) -- (1.5,-1.8);
				\draw[dashed] (-1.5,-1.3) -- (-1.5,-1.8);
				\draw (-.55,-2.9) -- (.55,-2.9);
				\draw[dashed] (-.55,-2.9) -- (-1,-2.7);
				\draw[dashed] (.55,-2.9) -- (1,-2.7);
				\node[shape=circle,fill=black, label=above:$\scriptstyle0$,scale=0.3] (0) at (0,0) {};
				\node[shape=circle,fill=black, label=above right:$\scriptstyle1$,scale=0.3] (1) at (1,-.4) {};
				\node[shape=circle,fill=black, label=above left:$\scriptstyle-1$,scale=0.3] (1) at (-1,-.4) {};
				\node[shape=circle,fill=black, label=left:$\scriptstyle-2$,scale=0.3] (1) at (-1.5,-1.3) {};
				\node[shape=circle,fill=black, label=right:$\scriptstyle2$,scale=0.3] (1) at (1.5,-1.3) {};
				\node[shape=circle,fill=black, label=below left:$\scriptstyle\tfrac{p+1}{2}$,scale=0.3] (1) at (-.55,-2.9) {};
				\node[shape=circle,fill=black, label=below right:$\scriptstyle\tfrac{p-1}{2}$,scale=0.3] (1) at (.55,-2.9) {};
	\end{tikzpicture}}} \quad\quad\quad\quad\quad
	P_0 = \vcenter{\hbox{\begin{tikzpicture}[scale=.5,vrtx/.style args = {#1/#2}{%
					circle, draw, thick, fill=white,
					minimum size=.1, label=#1:#2}]
				\draw (-1,-.4) -- (1,-.4);
				\draw (-1.5,-1.3) -- (1.5,-1.3);
				\draw[dotted] (1.5,-1.3) -- (1.5,-1.7);
				\draw[dotted] (-1.5,-1.3) -- (-1.5,-1.7);
				\draw (-.55,-2.9) -- (.55,-2.9);
				\draw[dotted] (-.55,-2.9) -- (-1,-2.7);
				\draw[dotted] (.55,-2.9) -- (1,-2.7);
				\node[shape=circle,fill=black, label=above:$\scriptstyle0$,scale=0.3] (0) at (0,0) {};
				\node[shape=circle,fill=black, label=above right:$\scriptstyle1$,scale=0.3] (1) at (1,-.4) {};
				\node[shape=circle,fill=black, label=above left:$\scriptstyle-1$,scale=0.3] (1) at (-1,-.4) {};
				\node[shape=circle,fill=black, label=left:$\scriptstyle-2$,scale=0.3] (1) at (-1.5,-1.3) {};
				\node[shape=circle,fill=black, label=right:$\scriptstyle2$,scale=0.3] (1) at (1.5,-1.3) {};
				\node[shape=circle,fill=black, label=below left:$\scriptstyle\tfrac{p+1}{2}$,scale=0.3] (1) at (-.55,-2.9) {};
				\node[shape=circle,fill=black, label=below right:$\scriptstyle\tfrac{p-1}{2}$,scale=0.3] (1) at (.55,-2.9) {};
	\end{tikzpicture}}}
	\]
	We let \(\F_p\) act naturally on the vertex set by addition, giving a rotation \(P_a\) of \(P_0\) for each \(a\in\F_p\).
	We similarly define \(C^\star\) and \(P_a^\star\) where the vertex set is the projective line \(\mathbb{P}^1(\F_p)=\F_p\cup\{\star\}\), and~\(\star\) has degree 0 in \(C^\star\) and is connected to \(-1\) in \(P_0^\star\).
\end{example}

\begin{proposition}
	For each odd prime \(p\), the set \(\{P_0,\dotsc,P_{p-1},C\}\) is a clique of graphs.
\end{proposition}
\begin{proof}
	Because of the action of the affine group \(\F_p\rtimes\F_p^*\) on the vertices, it suffices to consider the differences of \(P_0\) with \(C\) and \(P_{1/2}=P_{(p+1)/2}\).
	
	First consider \(C\). Write \(I=\left(\begin{smallmatrix} 1 & 0 \\ 0 & 1 \end{smallmatrix}\right)\), \(F=\left(\begin{smallmatrix} -1 & -1 \\ -1 & -1 \end{smallmatrix}\right)\), \(A=\left(\begin{smallmatrix} -1 & 0 \\ 0 & -1 \end{smallmatrix}\right)\) and \(B=\left(\begin{smallmatrix} 0 & -1 \\ -1 & 0 \end{smallmatrix}\right)\). Consider the \(2k\times 2k\)-block matrix 
	\[\alpha(k)=\begin{pmatrix}F & I & \\ I & F \\ & & \!\!\!\smash{\ddots}\!\!\! \\ & & & F & I \\ & & & I & F & I \\ & & & & I & A \end{pmatrix}\]
	and similarly \(\beta(k)\) with \(B\) in the place of \(A\). Using elementary row and column operations we obtain that
	\[\det(\alpha(k)) = \det\begin{pmatrix}\ddots\!\!\! \\ & F & I \\ & I & B &  \\ & &  & A \end{pmatrix} = \det(A)\cdot \det(\beta(k-1)) =\det(\beta(k-1))\]
	and similarly \(\det(\beta(k))=\det(B)\cdot\det(\alpha(k-1))=-\det(\alpha(k-1))\). It follows by induction that \(\det(\alpha(k)),\det(\beta(k))\in\{-1,1\}\) for all \(k\geq0\). We are done because \(\alpha(\frac{p-1}{2})\) is exactly \(\mathcal{L}(P_0-C)^0\) once we order the vertices \(0,1,-1,2,-2,3,-3,\dotsc\).
	
	Now consider \(P_{1/2}\). We obtain the graph
	\[P_0-P_{1/2} = \vcenter{\hbox{\begin{tikzpicture}[scale=.5,vrtx/.style args = {#1/#2}{%
					circle, draw, thick, fill=white,
					minimum size=.1, label=#1:#2}]
				\draw (0,0) -- (1,1) -- (1,-1) -- (2,1) -- (2,-1);
				\draw[dashed] (2,-1) -- (2.5,0);
				\draw[dotted] (3,0) -- (3.5,0);
				\draw[dashed] (4,0) -- (4.5,1);
				\draw (4.5,1) -- (4.5,-1);
				\node[shape=circle,fill=black, label=above:$\scriptstyle0$,scale=0.3] (0) at (0,0) {};
				\node[shape=circle,fill=black, label=above:$\scriptstyle1$,scale=0.3] (1) at (1,1) {};
				\node[shape=circle,fill=black, label=above:$\scriptstyle2$,scale=0.3] (2) at (2,1) {};
				\node[shape=circle,fill=black, label=above:$\scriptstyle\frac{p-1}{2}$,scale=0.3] (x) at (4.5,1) {};
				\node[shape=circle,fill=black, label=below:$\scriptstyle-1$,scale=0.3] (m1) at (1,-1) {};
				\node[shape=circle,fill=black, label=below:$\scriptstyle-2$,scale=0.3] (m2) at (2,-1) {};
				\node[shape=circle,fill=black, label=below:$\scriptstyle\frac{p+1}{2}$,scale=0.3] (y) at (4.5,-1) {};
	\end{tikzpicture}}},  \tag{1}\label{eq:graph}
	\]
	where the vertical edges are negative and the diagonal edges positive. This is clearly a tree, so \(P_0-P_{1/2}\) has a single spanning tree.
\end{proof}

The previous proposition provides only a minor improvement for matrices over a result from Lenstra (3.1 in \cite{Lenstra}) for number fields, namely that \(\Z[\zeta_p]\) contains a clique of size \(p\) for each prime \(p\).
For the following proposition, however, we are not aware of a comparable result from number theory.

\begin{proposition}
	For each odd prime \(p\), the set \(\{P_0^\star,\dotsc,P_{p-1}^\star,C^\star\}\) is a clique of graphs.
\end{proposition}
\begin{proof}
	By symmetry it suffices to consider the differences of \(P_0^\star\) with \(C^\star\) and \(P_{1/2}^\star\).
	
	For \(C^\star\), note that \(P_0^\star-C^\star\) is simply \(P_0-C\) with an additional vertex \(\star\) connected to a single vertex in \(\F_p\).
	Since every spanning tree of \(P_0^\star-C^\star\) must contain this edge, we have a bijection between the spanning trees of \(P_0^\star-C^\star\) and those of \(P_0-C\) that preserves signs.
	Hence we are done by the previous proposition.
	
	Now consider \(P_{1/2}^\star\). 
	As seen in (\ref{eq:graph}), \(P_0-P_{1/2}\) is a line of edges of alternating signs. Then \(P_0^\star-P_{1/2}^\star\) is this same line with additionally the two points \(-1\) and \(1/2-1=(p-1)/2\) connected through \(\star\). 
	Hence \(P_0^\star-P_{1/2}^\star\) has a single cycle, which we observe has odd length, since the distance between \(-1\) and \((p-1)/2\) in \(P_0-P_{1/2}\) is odd.
	All spanning trees of \(P_0^\star-P_{1/2}^\star\) are obtained by removing a single edge of its unique cycle, and since the (absolute) difference between the number of positive edges and negative edges on this cycle is \(1\), we may conclude similarly that the (absolute) difference between the number of spanning trees of positive sign and that of negative sign is \(1\), as was to be shown.
\end{proof}

We combine the propositions into the following theorem.

\begin{theorem}\label{thm:graphs}
	For all odd primes \(p\) the rings \(\textup{Mat}_{p-1}(\Z)\) and \(\textup{Mat}_p(\Z)\) both contain an exceptional clique consisting of \(p+1\) symmetric matrices with coefficients in \(\{-2,-1,0,1\}\). \qed
\end{theorem}

\section{Rank 2 cliques}\label{sec:rank_2}

In this section we will show that up to equivalence \(\textup{Mat}_2(\Z)\) contains only three maximum exceptional cliques.
We define the matrices \(\phi=\big(\begin{smallmatrix}0&1\\1&1\end{smallmatrix}\big)\) and \(\zeta=\big(\begin{smallmatrix}0&1\\-1&1\end{smallmatrix}\big)\).

\begin{lemma}\label{lem:integer_conjugacy}
	Every exceptional unit in \(\textup{Mat}_2(\Z)\) is conjugate to either \(\zeta\), \(\phi-1\), \(\phi\) or \(\phi+1\) and is equivalent to either \(\zeta\) or \(\phi\).
\end{lemma}
\begin{proof}
	As seen in Example~\ref{ex:phi}, each exceptional unit \(u\in\textup{Mat}_2(\Z)\) shares a characteristic polynomial with either \(\zeta\), \(\phi-1\), \(\phi\) or \(\phi+1\). That they are  conjugate then follows, in the case of \(\zeta\), from \cite[Theorem 11.2]{matrix_conjugacy}; and, in the case of \(\phi\), and consequently for \(\phi-1\) and \(\phi+1\), it follows from \cite[ Example~9.10, Corollary~10.10 and Theorem~11.8]{matrix_conjugacy}. The equivalence also follows from our Example~\ref{ex:phi}.
\end{proof}

\begin{lemma}\label{lem:test_clique_equivalence}
	Suppose \(\crg\) is a commutative ring. 
	For an exceptional clique \(\clq\subseteq \crg\), write \([\clq]\) for the multiset of elements \((s-t)^{-1}\cdot (u-v)\) for all pairwise distinct \(s,t,u,v\in \clq\).
	If exceptional cliques \(\clq,\clqb\subseteq \crg\) are equivalent, then there exists some \(\sigma\in\Aut(\crg)\) such that \(\sigma[\clq]=[\clqb]\). \qed
\end{lemma}

\begin{proposition}\label{prop:dim2_conj}
	Each exceptional clique of size \(4\) in \(\textup{Mat}_2(\Z)\) is equivalent to either \(\{0,1,\phi,\phi+1\}\), \(\{0,1,\zeta,1-\phi\}\) or \(\{0,1,\zeta,\phi-1\}\).
\end{proposition}
\begin{proof}
	By Lemma~\ref{lem:test_clique_equivalence} the clique \(\clq_1=\{0,1,\phi,\phi+1\}\) is not equivalent to both \(\clq_2=\{0,1,\zeta,1-\phi\}\) and \(\clq_3=\{0,1,\zeta,\phi-1\}\) because \(1\in[\clq_1]\) yet \(1\not\in[\clq_2]\) and \(1\not\in[\clq_3]\). The remaining two cliques are not equivalent, since \([\clq_2]\) consists of \(6\) and \([\clq_3]\) of \(10\) distinct elements.
	
	Suppose \(\{0,1,u,v\}\) is an exceptional clique of size \(4\) and write \(d_t=\det(v-t)\in\{-1,1\}\) for \(t\in\{0,1,u\}\).
	By Lemma~\ref{lem:integer_conjugacy}, we may assume \(u=\big(\begin{smallmatrix}0&1\\s&1\end{smallmatrix}\big)\) for \(s\in\{1,-1\}\).
	Let \(g\in\textup{GL}_2(\Z)\) be such that \(gug^{-1}\) is the Galois conjugate of \(u\) in \(\Q(u)\) and consider the map \(\sigma\) given by \(x\mapsto 1-gxg^{-1}\).
	Note that \(\sigma\) fixes \(\{0,1,u\}\) and acts on the determinants as \((d_0,d_1,d_u)\mapsto (d_1,d_0,d_u)\).
	Write \(v=\big(\begin{smallmatrix}a&b\\c&d\end{smallmatrix}\big)\).

	With 
	\[\delta=d_1-d_0-1=-a-d \quad\text{and}\quad \varepsilon=d_u-d_0+s=-a+sb+c\] 
	we have
	\[a^2+ab-sb^2+\delta a + \varepsilon b + d_0=b(a-sb+\varepsilon)-a(-\delta-a) + d_0 = -(ad-bc)+d_0 = 0 \tag{2}\label{eq:X}.\]
	Note that \(\delta\) and \(\varepsilon\) are odd.
	It follows that \(a(\delta+a)+d_0\) is odd, hence \(b(a-sb+\varepsilon)\) must be odd.
	In particular, \(b\) and consequently \(a\) must be odd.
	
	(\emph{Case \(u=\zeta\)})
	Suppose \(s=-1\).
	Consider the map \(\tau\) given by \(x\mapsto1-u^{-1}x\) and note that it fixes \(\{0,1,u\}\) and acts on the determinants as \((d_0,d_1,d_u)\mapsto(d_u,d_0,d_1)\). Together with the action of \(\sigma\), we may assume \((d_0,d_1,d_u)\in\{(+1,+1,+1),(-1,+1,+1),(-1,-1,+1),(-1,-1,-1)\}\).
	From (\ref{eq:X}) we obtain
	\begin{align*}
		\Big(\frac{a+b}{2}\Big)^2 + \Big(\frac{a+\delta}{2}\Big)^2 + \Big(\frac{b+\varepsilon}{2}\Big)^2 = \frac{-2d_0+\delta^2+\varepsilon^2}{4} = 
		\begin{cases}
			0 & \textup{if } d_0=+1 \\
			1 & \textup{if } d_0=-1
		\end{cases},
	\end{align*}
	where each term on the left hand side is a non-negative integer. This puts on sufficient restrictions to efficiently enumerate the solutions for \(v\):
	\begin{align*}
		\begin{pmatrix}
			-1 & -1 \\
			-1 & \phantom{-}0
		\end{pmatrix},
		\begin{pmatrix}
			-1 & \phantom{-}1 \\
			\phantom{-}1 & \phantom{-}0
		\end{pmatrix},
		\begin{pmatrix}
			\phantom{-}1 & -1 \\
			\phantom{-}1 & -2
		\end{pmatrix},
		\begin{pmatrix}
			\phantom{-}1 & \phantom{-}1 \\
			\phantom{-}1 & \phantom{-}0
		\end{pmatrix},
		\begin{pmatrix}
			\phantom{-}1 & -1 \\
			-1 & 0
		\end{pmatrix},
		\begin{pmatrix}
			-1 & \phantom{-}1 \\
			-1 & \phantom{-}2
		\end{pmatrix}.
	\end{align*}
	The map \(1\mapsto u x u^{-1}\) fixes \(\{0,1,u\}\) and transitively permutes the first \(3\) and the last \(3\) matrices. Hence each gives a clique equivalent to either \(\{0,1,\zeta,1-\phi\}\) or \(\{0,1,\zeta,\phi-1\}\).
	
	(\emph{Case \(u=\phi\)}) Suppose \(s=1\). If \((d_0,d_1)=(+1,+1)\), then \(v^2-v+1=0\), so \(v\) is conjugate to \(\zeta\) and we are in the previous case. Similarly, if \((d_0,d_u)=(-1,-1)\), then \(u^{-1}\{0,1,u,v\}\) contains \(u^{-1}v\), which is conjugate to \(\zeta\). Together with the action of \(\sigma\), only the cases \((d_0,d_1,d_u)\in\{(-1,-1,+1),(+1,-1,+1)\}\) remain.
	In the former case, (\ref{eq:X}) becomes
	\[0 = a^2+ab-b^2-a+3b-1 \equiv ab-a-b-1 \equiv (a-1)(b-1)-2 \ (\textup{mod }4),\]
	which has no solutions for odd \(a\) and \(b\). In the latter case, we obtain
	\[a^2+ab-b^2-3a+b+1 = \big((a-1)+\phi (b-1)\big)\big((a-1)+(1-\phi)(b-1)\big)=N\big((a-1)+\phi(b-1)\big).\]
	Hence \(a=b=1\), which corresponds to the solution \(v=\phi+1\).
\end{proof}

\section{Miscellanea}\label{sec:miscellanea}

In this section we present various independent results.

So far we have presented examples that show for \(n=1,2,4\) that the upper bound of \(2^n\) on the clique number of \(\textup{Mat}_n(\Z)\) can be achieved. This result was already known for \(n\leq 4\), although the cliques produced by \cite{BC24} do not follow a conceptual construction. For \(n=4\), we obtain the optimal clique from a general construction in Example~\ref{ex:n4c16}. Here we present our attempt for \(n=3\), which does not seem to generalize to other dimensions.

\begin{example}\label{ex:n3c8}
	The ring \(\textup{Mat}_3(\Z)\) contains an exceptional clique of size \(8\).
	
	Let \(C_3=\langle(1\ 2\ 3)\rangle\) be the cyclic group of order \(3\) and consider its action on \(\textup{Mat}_3(\Z)\) by permuting the standard basis vectors. Let 
	\[M = \left(\begin{smallmatrix} 1 & 0 & -1 \\ 0 & 0 & 1 \\ -1 & 1 & 0\end{smallmatrix}\right) \in \textup{GL}_3(\Z).\]
	It is interesting to notice that \(\Z[M]\cong\Z[\zeta_7+\zeta_7^{-1}]\) is a cyclotomic ring, although we have not been able to do much with it.
	We claim that the image of \(\{0,1\}^{C_3}\) under the map \(\Z^{C_3}\to\textup{Mat}_3(\Z)\) given by 
	\[(a_\sigma)_{\sigma\in C_3} \mapsto \sum_{\sigma\in C_3} a_{\sigma} \cdot M^\sigma\] 
	is a maximal exceptional clique.
	It suffices to show that the image of \(\{-1,0,1\}^{C_3}\setminus\{0\}\) is in \(\textup{GL}_3(\Z)\).
	Both \(C_3\) and \(\Z^*=\{-1,1\}\) act on \(\{-1,0,1\}^{C_3}\setminus\{0\}\), so that it remains to only verify the determinants of
	\[ M\pm M^{(1\ 2\ 3)} \quad\textup{and}\quad M+M^{(1\ 2\ 3)}\pm M^{(1\ 3\ 2)}.\]
	Thus one can verify this is an exceptional clique by computing \(5\) determinants (including that of $M$) instead of \(\binom{8}{2}=28\). A direct computation shows their determinants are indeed $\pm 1$. 
	This approach does not seem to generalize well to higher dimensions if we pick a random matrix $M$ (with small coefficients), as the probability to find a good candidate is quite low. Perhaps it is necessary to use more structure (e.g. that of $\Z[M]$) but we were not able to do so.
\end{example}

So far we have not considered constructing exceptional cliques in orders of arbitrary degree. In the following two examples we provide heuristic constructions. 

\begin{example}\label{ex:general_nf_construction}
%	Fix \(m\in\Z_{\geq0}\) and let \(f\in\Z[X]\) be monic such that every non-zero polynomial of degree less than \(m\) and with coefficients in \(\{-1,0,1\}\) divides \(f\). Clearly we may construct such an \(f\) of every sufficiently large degree. Moreover, with large probability there exists for every sufficiently large degree an \(f\) for which \(h=f-1\) is irreducible. For \(m=3\), for example, \(f\) is a multiple of
%	\[X(X^2 - 1) (X^2 - X - 1) (X^2 - X + 1) (X^2 + 1) (X^2 + X - 1) (X^2 + X + 1).\]
%	Then with \(\Z[\alpha]=\Z[X]/(h)\) it holds that %\(\sum_{i=0}^{m-1} s_i \alpha^i \in \Z[\alpha]^*\) for all \((s_i)_i \in \{-1,0,1\}^m\) that are not the zero vector. Hence \(\Z[\alpha]\) contains the exceptional clique \(\{\sum_{i=0}^{m-1} t_i \alpha^i \,|\, (t_i)_i \in \{0,1\}^m\}\) of size \(2^m\).
%	Hence we may heuristically compute for every \(3^m\leq n \leq 3^{m+1}\) \inote{Where does this come from? I have naively computed the degree of the product of all monic polynomials of degree $<m$, removing the ones with independent coefficient 0 (except for X) and get $\frac{(2m-3)\cdot 3^{m-1}-1}{2}\leq m3^ {m-1}$ but I suppose one can do better?} an order of degree \(n\) with a clique of size \(2^m\approx n^{(\log 2) / (\log 3)}\).

%\textbf{Milan: I propose the following discussion+proposition as a replacement of the previous example}\\
Let $K$ be a number field and $C\subset\ent_K$ an exceptional clique with $0\in C$ and define $\Delta C=\{c-c'\mid c,c'\in C\}$. If $E\subset K$ is some subset, write $E[X]_{<m}$ for the set of polynomials in $K[X]_{<m}$ with coefficients in $E$. Let $P_m$ be the $\lcm$ of all irreducible factors over $\ent_K$ of all non-zero polynomials in $\Delta C[X]_{<m}$. Then, the set $C[X]_{<m}$ forms an exceptional clique of size $\# C^m$ in the order $\ent_K[X]/(P_m-1)$ which has rank $\deg P_m$. Notice that for any rank $r\ge \deg P_m$ one can replace $P_m$ by $X^{r-\deg P_m}P_m$ to obtain a clique of same size but in an order of rank $r$.\\
~\\
Let us apply the ideas of the previous example with $K=\Q$ and $C=\{0,1\}$. Here $\Delta C=\{-1,0,1\}$. We need to bound the degree of $P_m$. Let $E_m$ be the set of monic polynomials in $\Delta C[X]_{<m} $ with non-zero constant coefficient. One sees that $P_m\mid X\cdot\prod_{Q\in E_m} Q$ so we get the upper bound $$\deg P_m\le 1+\sum_{i=1}^{m-1} 2\cdot i\cdot 3^{i-1}\le m3^{m-1}.$$ Corollary 3 of \cite{breuillard_varju_2019} proves that under GRH, the probability that a polynomial in $E_m$ is irreducible is $1-O(m^{-\frac{1}{2}})$. This implies the asymptotic lower bound $$\deg P_m\ge (1-o(1))\cdot m3^{m-1},$$ so our upper bound is asymptotically tight.\\
We sum up the previous discussion in the following result.
\end{example}
\begin{proposition}
For any $n\in\Z_{>0}$, there exists a bound $N\in\Z_{>0}$ such that for any integer $r\ge N$ there exists an order of rank $r$ containing an exceptional clique of size $\ge n$. Moreover one can take $$N\le \lceil\log_2(n)\rceil\cdot n^{\log_2(3)}.$$
\end{proposition}

\begin{example}
	Let \(f\in\Z[\phi][X]\) (where as before $\phi$ is a root of $X^2-X-1$) be a monic polynomial such that \(X-a/b\) divides \(f\) for all \(a,b\in\{-1,0,1\}+\{-\phi,0,\phi\}\) and \(b\neq 0\). The least degree of such an \(f\) is \(15\). If we again assume \(f-1\) is irreducible, we obtain a clique of size \(16\) given by \(\{aX+b\,|\, a,b \in \{0,1,\phi,\phi+1\}\}\subseteq\Z[\phi][X]/(f-1)\). In particular, this applies to a number field of degree \(30\).
\end{example}

The following proposition shows that large exceptional cliques in \(\textup{Mat}_n(\Z)\) must generate simple \(\Q\)-algebras.

\begin{lemma}\label{lem:Q_to_Z}
	Let \(Z\) be a principal ideal domain with field of fractions \(Q\). 
	Let \(\nrg\subseteq\textup{Mat}_n(Q)\) be a \(Z\)-subalgebra that is finitely generated as \(Z\)-module. 
	Then there exists an injective ring homomorphism \(\nrg\to\textup{Mat}_n(Z)\). \qed
\end{lemma}
%\begin{proof}
%Let \(M=R\cdot Z^n\subseteq Q^n\) and note that \(M\) is a finitely generated torsion-free \(Z\)-module. Hence \(M\) is free of rank at most \(n\) by the structure theorem of finitely generated modules over principal ideal domains. The map \(R\to \textup{End}_Z(M)\) is injective and \(\textup{End}_Z(M)\cong\textup{Mat}_n(Z)\).
%\end{proof}

\begin{proposition}\label{prop:clique_implies_simple}
	Let \(Z\) be a principal ideal domain with field of fractions \(Q\) and let \(\m\subseteq Z\) be a maximal ideal.
	Let \(\nrg\subseteq\textup{Mat}_n(Z)\) be a \(Z\)-subalgebra that is finitely generated as \(Z\)-module. 
	If \(\nrg\) contains an exceptional clique of size larger than \(\#(Z/\m)^{n-1}\), then \(Q\tensor_Z \nrg\) is a simple \(Q\)-algebra.
\end{proposition}
\begin{proof}
	Write \(M=\textup{Mat}_n(Q)\) and \(A=Q\tensor_Z \nrg\) and let \(V\) be a simple left \(M\)-module.
	Suppose there is some non-zero \(A\)-submodule \(U \subseteq V\) of dimension \(d\) over \(Q\).
	Then we have a ring homomorphism \(\nrg\to\textup{End}(U)\) and in turn a homomorphism \(\nrg\to\textup{Mat}_d(Z)\) by Lemma~\ref{lem:Q_to_Z}.
	As \(\nrg\) has a clique of size larger than \(\#(Z/\m)^{n-1}\), it follows from Proposition~\ref{prop:injective_quotient} that \(n\leq d\).
	Hence \(V\) is a simple \(A\)-module, \(\textup{Jac}(A)\subseteq\textup{Jac}(M)=0\), and \(A\) is semisimple. 
	Write \(A\cong \prod_{i=1}^t \textup{Mat}_{m_i}(D_i)\) as in Theorem~\ref{thm:artin-wedderburn}, so that with \(d_i=\dim_Q(D_i)\) we have \(\sum_{i} m_i d_i \leq n\).
	%If \(k\geq 2\), then \(\Q^{m_1}\times \Q\)
	% such that \(\sum_{i=1}^k n_i d_i = n\). 
	We have homomorphisms \(\textup{Mat}_{m_i}(D_i)\to\textup{Mat}_{m_i d_i}(Q)\) and in turn \(\nrg\to\textup{Mat}_{m_id_i}(Z)\) by Lemma~\ref{lem:Q_to_Z}.
	Hence \(n\leq m_id_i\) by Proposition~\ref{prop:injective_quotient}.
	We conclude that \(t=1\) and thus \(A\) is simple.
\end{proof}

Choosing $Z=\Z$ and $\m=(2)$ gives the following corollary.

\begin{corollary}
	Let $\nrg\subseteq\textup{Mat}_n(\Z)$ be a $\Z$-subalgebra. If $\nrg$ contains an exceptional clique of size larger than $2^{n-1}$, then $\Q\tensor_Z\nrg$ is a simple $\Q$-algebra.
\end{corollary}

\section{Table of records}\label{sec:table_of_records}

In this section we present a table of the sizes of the largest cliques known to us in small dimension, both in orders and in matrix rings. \\

We briefly explain the notations. For \(n\in\Z_{>0}\) let \(\zeta_n\) be a primitive \(n\)-th root of unity and \(\eta_n=\zeta_n+\zeta_n^{-1}\). We abbreviate the number field \(\Q[X]/(X^n+f_{n-1}X^{n-1}+\dotsm+f_0)\) by \((f_{n-1},\dotsc,f_0)\).
We denote by \(\A(L/K,u)\) a clique in \(\mathcal{A}_{L/K}\) obtained from Proposition~\ref{prop:nt_square_clique} with the parameter \(u\in\mathcal{O}_L\), and let \(\mathcal{B}(L/F/K,u)\) denote a clique in the twisted algebra \(\mathcal{A}_{L/K}^u\) obtained from Lemma~\ref{lem:mat_embedding_twist} and Proposition~\ref{prop:nt_square_clique_twisted} with $\eta=N_{L/F}(u)$.  We write $\Mat_n\oplus\Mat_m$ for the construction using Lemma \ref{lem:trivial_induction}.

\begin{align*}
	\begin{array}{r|lr|lr}
		n & \textup{number field of degree \(n\)} & \#\clq & \textup{Mat}_n(\Z) & \# \clq \\ \hline
		2 & \Q(\eta_5) & 4 & & 4 \\ 
		3 & \Q(\eta_7) & 7 & \text{Example~\ref{ex:n3c8}} & 8 \\
		4 & \text{Example~\ref{ex:nf_deg_4}} & 10 & \mathcal{B}\big(\Q(\eta_5)/\Q(\eta_5)/\Q,\phi\big) & 16 \\
		5 & \Q(\eta_{11}) & 11 & & 11 \\ 
		6 & \Q(\eta_{5}, \eta_{7}) & 18 & \A\big(\Q(\eta_5,\eta_7)/\Q(\eta_5),\eta_7-\eta_5\big) & 49 \\
		7 & (0,-3,-1,1,3,1,-1) & 17 & & 17 \\ 
		8 & \Q(\eta_{17}) & 16 & \mathcal{B}\big(\Q(\zeta_5)/\Q(\eta_5)/\Q,\phi\big)  & 25 \\
		9 & \Q(\eta_{19}) & 18 & \mathcal{B}\big(\Q(\eta_7)/\Q(\eta_7)/\Q,\eta_7\big)  & 49 \\
		10& (-2,-1,4,-3,-4,7,4,-5,-1,1) & 20 & \textup{Theorem~\ref{thm:clique_doubling}}  & 22 \\
		11& \Q(\eta_{23}) & 22 &   & 22 \\
		12& \Q(\eta_{5\cdot 7}) & 22 & \A\big(\Q(\eta_5,\eta_7)/\Q,\phi\big)  & 324 \\
		13& (-1,1,-4,5,-5,6,-2,-5,5,-7,5,-1,1) & 21 & & 21 \\
		14& \Q(\eta_{29}) & 28 & \textup{Theorem~\ref{thm:clique_doubling}} & 34 \\
		15& \Q(\eta_{31}) & 30 & \A\big(\Q(\eta_7,\eta_{11})/\Q(\eta_7),\eta_{11}-\eta_7\big) & 121 \\
		16& \Q(\zeta_{17}) & 17 & \A\big(\Q(\eta_5,\eta_{17})/\Q(\eta_5),\eta_{17}-\eta_5\big) & 256 \\
		17& \text{Example~\ref{ex:general_nf_construction}} & 8 & \textup{Mat}_8 \oplus \textup{Mat}_9 & 25 \\
		18& \Q(\eta_{37}) & 36 & \mathcal{B}\big(\Q(\eta_5,\eta_7)/\Q(\eta_5,\eta_7)/\Q(\eta_7),\eta_7\big) & 324 \\
		19& \text{Example~\ref{ex:general_nf_construction}}  & 8 &  \textup{Mat}_8 \oplus \textup{Mat}_{11}  & 22 \\
		20& \Q(\eta_{41}) & 40 & \A\big(\Q(\zeta_{33})/\Q(\zeta_3),\zeta_6\big) & 121 \\
	\end{array}
\end{align*}

\section{Exceptional cliques in $\Hom_\crg(M,N)$ }\label{sec:rankmetric}

In this section we generalize the definition of an exceptional clique. The motivation for this stems again from applications to cryptography, see Section~\ref{sec:exceptional_cliques_in_cryptography}.

\begin{definition}
	Let \(\crg\) be a commutative ring and let \(M\) and \(N\) be \(\crg\)-modules.
	A subset \(\clq\subseteq\Hom_\crg(M,N)\) is an \emph{exceptional clique} if for all distinct \(f,g\in \clq\) the map \(f-g\) has a left inverse.
\end{definition}

If we take \(\crg=\Z\) and \(M=N=\Z^n\) for some \(n\geq 0\), this agrees with the definition of an exceptional clique in \(\textup{Mat}_n(\Z)\), while if we take \(M=N=\crg\) for some commutative ring \(\crg\) this agrees with the definition of an exceptional clique in \(\crg\). 
As in Corollary~\ref{cor:clique_bound}, the size of an exceptional clique in \(\Hom_\Z(\Z^m,\Z^n)\) is bounded in size by \(2^n\).

\begin{lemma}\label{lem:local_left_inverse}
	Let \(\crg\) be a Noetherian commutative ring and \(f:M\to N\) a morphism of Noetherian \(\crg\)-modules. Then \(f\) has a left inverse if and only if for every maximal ideal \(\m\subset \crg\) the localization \(f_\m:M_\m\to N_\m\) has a left inverse.
\end{lemma}
\begin{proof}
	The condition that \(f\) has a left inverse is equivalent to the map \(\Hom_\crg(N,M)\to\Hom_\crg(M,M)\) given by \(g\mapsto g\circ f\) being surjective. Since surjectivity is a local property, it remains to show that the natural map \(\Hom_\crg(A,B)_\m\to\Hom_{k_\m}(A_\m,B_\m)\) is an isomorphism for every maximal \(\m\subset \crg\) and all Noetherian \(\crg\)-modules \(A\) and \(B\). By the assumptions on \(A\) and \(\crg\) there exists an exact sequence \(\crg^s\to \crg^t\to A\to 0\). We obtain the following commutative diagram 
	\begin{center}
		\begin{tikzcd}[row sep = 1.4em]
			0 \arrow{r}& \Hom_\crg(A,B)_\m \arrow{r}\arrow{d}& B^t_\m \arrow{r}\arrow{d}& B^s_\m \arrow{d} \\
			0 \arrow{r}& \Hom_{\crg_\m}(A_\m,B_\m) \arrow{r}& B^t_\m \arrow{r}& B^s_\m
		\end{tikzcd}
	\end{center}
	where the exact rows are obtained by applying the Hom and localization functors to the original sequence in different orders, and using the natural isomorphism \(\Hom_\crg(\crg^t,B)\cong B^t\). The claim now follows from the five lemma.
\end{proof}

%\begin{definition}
%Let \(k\) be a non-zero commutative ring and \(f:M\to N\) a morphism of \(k\)-modules. We define the \emph{projective rank} of \(f\) to be
%\[\textup{proj rk}(f) = \max\{ \rk(A) \,|\, A\subseteq \im(f)\textup{ a free \(k\)-module and a summand of \(N\)} \}.\]
%\end{definition}

\begin{definition}
	Let \(k\) be a commutative ring and \(\mathfrak{a}\subseteq R\) an ideal of a commutative \(k\)-algebra. 
	For each maximal ideal \(p\subset k\) we write \(v_p(\mathfrak{a})\) for \(\dim_{k/p}R/(\mathfrak{a}+pR)\), or simply \(v_p(a)\) when \(\mathfrak{a}=aR\) for some \(a\in R\). 
	%Equivalently, it is the rank of \(\prod_\m R_\m\) where the product ranges over all maximal ideals of \(R\) not containing \(\mathfrak{a}\).
\end{definition}

%\begin{lemma}\label{lem:rk_det_rel}
%Let \(p\subset k\) be a maximal ideal an let \(\mathfrak{a}\subseteq R\). Then \(v_p(\mathfrak{a})\geq n - \textup{ord}_p\, N(\mathfrak{a})\).
%Let \(K\) be a number field and \(\mathfrak{a}\subseteq \mathcal{O}_K\) be a non-zero ideal. 
%Then \(v_p(\mathfrak{a})\leq \textup{ord}_p\, N_{K/\Q}(\mathfrak{a})\).
%\end{lemma}
%\begin{proof}
%It holds that \(\#(R/(\mathfrak{a}+pR)) = N(\mathfrak{a}+pR) \mid  N(\mathfrak{a})\).
%\(\mathcal{O}_K/\mathfrak{a}\twoheadrightarrow \mathcal{O}_K/(\mathfrak{a}+p\mathcal{O}_K) \cong \prod_\m (\mathcal{O}_K/\m) \), where the product ranges over all maximal ideals \(\m\subset \mathcal{O}_K\) above \(p\) containing \(\mathfrak{a}\).
%Hence \(\prod_\m \#(\mathcal{O}_K/\m) \mid N(\mathfrak{a})\).
%\end{proof}

\begin{proposition}\label{prop:rank_metric_general}
	Let \(k\) be a Dedekind domain and \(R=k[x]\) be a monogenic commutative \(k\)-algebra that is free of rank \(n\) as \(k\)-module. Let \(0<m\leq n\).
	If \(\clq\subseteq R\) is such that \(v_p(u-v)\leq n-m\) for all distinct \(u,v\in \clq\) and maximal ideal \(p\subset k\), then \(\clq\) maps injectively to \(\Hom_k(kx^0+\dotsm+kx^{m-1},R)\) and the image is an exceptional clique.
\end{proposition}
\begin{proof}
	Write \(M=kx^0+\dotsm+kx^{m-1}\). We have \(M\oplus(kx^m+\dotsm+kx^{n-1})=R\) and let \(\pi_M:R\to M\) and \(\iota_M:M\to R\) be the corresponding projection and inclusion.
	Let \(u,v\in \clq\) be distinct and \(a=u-v\). We will show that the map \(f=a\cdot \iota_M:M\to R\) has a left inverse.
	By Lemma~\ref{lem:local_left_inverse} we may assume \(k\) is local with maximal ideal \(p\). 
	Write \(A=\prod_\m R_\m\), a summand of \(R\), where the product ranges over all primes \(\m\subset R\) not containing \(a\), and let \(\pi_A:R\to A\) and \(\iota_A:A\to R\) be the corresponding maps.
	Consider the map \(g=a^{-1} \cdot \pi_A : R\to A\), which is well-defined since \(a\) is a unit in \(A\). 
	Since multiplication by \(a\) commutes with \(\pi_A\), we conclude that \(g\circ f=\pi_A\circ\iota_M\). It remains to show that the latter map has a left inverse.
	This follows readily from the fact that \(A=R/Q\) for some monic polynomial \(Q\) in \(x\) of degree \(n-v_p(a)\geq m\).
\end{proof}

\begin{lemma}\label{lem:norm_bound}
	Let \(p\) be an odd prime and \(\zeta\) a primitive \(p\)-th root of unity. For all \(0 \leq n \leq p\) we have
	\[\max\Big\{ \Big|N_{\Q(\zeta)/\Q}\Big( \sum_{i=0}^{n-1} a_i\zeta^i \Big)\Big| : a_i\in\{-1,0,1\} \Big\} \leq \Big(\frac{np}{p-1}\Big)^{(p-1)/2}. \]
\end{lemma}
\begin{proof}
	For \(p\nmid k\), write \(\sigma_k\in\textup{Aut}(\Q(\zeta))\) for \(\zeta\mapsto \zeta^k\).
	Recall that for all \(i\in\Z\) it holds that
	\[ \sum_{k=0}^{p-1} \zeta^{ki} = \begin{cases} p & \textup{if } p \mid i  \\ 0 & \textup{otherwise} \end{cases}. \]
	Let \(x=a_0\zeta^0+\dotsm+a_{n-1}\zeta^{n-1}\) with \(a_0,\dotsc,a_{n-1}\in\{-1,0,1\}\).
	Then
	\[ \sum_{k=0}^{p-1} |\sigma_k(x)|^2 = \sum_{k=0}^{p-1} \sigma_k(x\overline{x}) = \sum_{k=0}^{p-1} \sum_{i=0}^{n-1}\sum_{j=0}^{n-1}  a_i a_j \zeta^{k(i-j)} = \sum_{i=0}^{n-1}\sum_{j=0}^{n-1} a_i a_j \sum_{k=0}^{p-1} \zeta^{k(i-j)} = \sum_{i=0}^{n-1} a_i^2 p \leq np.\]
	Combined with the inequality of the arithmetic and geometric mean we obtain
	\[ |N_{\Q(\zeta)/\Q}(x)| = \prod_{k=1}^{p-1} |\sigma_k(x)| \leq \left(\frac{1}{p-1}\sum_{k=1}^{p-1} |\sigma_k(x)|^2 \right)^{(p-1)/2} \leq \Big(\frac{np}{p-1}\Big)^{(p-1)/2},\]
	as was to be shown.
\end{proof}

A prime \(p\) is \emph{safe} if \(p\) is odd and \((p-1)/2\) is prime. It is conjectured that there are infinitely many safe primes \cite{Shoup}, and if true the following theorem applies to infinitely many \(n\).

\begin{theorem}\label{thm:rank-metric}
	For every safe prime \(n+1\) and integer \(0<m\leq n/2\) there exists an exceptional clique in \(\Hom(\Z^m,\Z^n)\) of size \(2^n\).
\end{theorem}
\begin{proof}
	Let \(R=\Z[\zeta]\) where \(\zeta\) is a \((n+1)\)-th primitive root of unity. 
	We will apply Proposition~\ref{prop:rank_metric_general} to 
	\[ \clq=\Big\{ \sum_{i=0}^{n-1} a_i \zeta^i : a_i\in\{0,1\} \Big\}. \]
	Thus it suffices to show for all primes \(q\) and non-zero \(x=a_0\zeta^0+\dotsm+a_{n-1}\zeta^{n-1}\) with \(a_0,\dotsc,a_{n-1}\in\{-1,0,1\}\), that \(v_q(x)\leq n/2\).
	If \(q\geq n+1\), then by Lemma~\ref{lem:norm_bound} with \(p=n+1\) we have 
	\[v_q(x)\leq \log_q N(xR+qR)\leq \log_q N(x)\leq \log_q (n+1)^{n/2}\leq n/2.\]
	Thus suppose \(q<n\). For each maximal ideal \(q\in \m\subset R\) the multiplicative order of \(q\textup{ mod }n+1\) equals \(\dim_{\F_q} R/\m\). Since \(q\not\equiv \pm 1\ (\textup{mod }n+1)\) and \(n\) is safe, this order can only be \(n/2\) or \(n\).
	Since \(x\not\in qR\), there is some prime \(\m\subset R\) above \(q\) such that \(x\not\in\m\).
	Hence \(v_q(x)\leq n-(n/2) = n/2\).
\end{proof}

\section{Exceptional cliques in cryptography}\label{sec:exceptional_cliques_in_cryptography}

Exceptional units have several applications in cryptography. First, we consider the notion of secret sharing scheme, a well-known and well-studied cryptographic primitive. A secret sharing scheme allows a \textit{dealer} to distribute the knowledge of a secret among a set of parties (\textit{share receivers}) by sending each of them a related piece of information (a \textit{share}), in such a way that only prescribed subsets of these parties can later reconstruct the secret when they pool their received shares together, while other subsets can gather no information about the secret. For example, in a $t$-threshold secret sharing scheme, any set of at least $t+1$ shares allows one to determine the secret entirely, while any set of $t$ shares is jointly distributed independently from the secret. Over a finite field $\F$, Shamir secret sharing \cite{Sha79} is the best known way to realize a $t$-threshold secret sharing scheme. For a number $n$ of share receivers, under the restriction $n<|\F|$, the scheme requires fixing $n+1$ publicly known pairwise distinct elements of the field, say $\alpha_0,\alpha_1,\dots,\alpha_n$. The dealer chooses a polynomial uniformly at random in the set of polynomials $f$ in $\F[X]$ of degree at most $t$ and such that $f(\alpha_0)=s$. The shares are then defined as $f(\alpha_i)$, for $i=1,\dots,n$. The aforementioned properties are then both derived from the fact that a unique polynomial of degree at most $t$ interpolates $t+1$ points. It is not hard to verify that this in turn relies on the fact that all differences $\alpha_i-\alpha_j$ are invertible (since the $\alpha_i$ are pairwise different and we are working over a field).  

It is interesting to find counterparts of this scheme that can work over other algebraic structures of relevance in cryptography. In particular, the notion of \textit{black-box secret sharing} \cite{DF94} considers secret sharing schemes over a finite commutative group $G$ whose structure we may not know in full\footnote{e.g. $(\Z/N\Z)^*$ where $N$ is known, but it is a product of two large unknown primes}, but where we have some description of group elements and can compute the group operation and inverses on them (i.e. we have ``black-box access'' to these operations) and sample elements from $G$ uniformly at random. Using the $\Z$-module structure of the group one can regard this as a secret sharing problem on integers. Concretely, we could think of finding an analogue to the Shamir secret sharing above which consists of sampling a random polynomial in $G[X]$ and ``evaluating" the polynomial on integers using the action of $\Z$ on $G$.

However, the largest exceptional clique over the integers has two elements, which prevents us from obtaining the threshold property described above by these means. Techniques for constructing black-box secret sharing have thus either relied on embedding $\Z$ into a ring with a larger exceptional clique, or on using modified constructions which require slightly weaker properties from the set of evaluation points. In the first direction, \cite{DF94} use number fields as an embedding ring, but this leads to the size of each share being $\Theta(n)$ elements of the group (on average). The second approach yields better results: \cite{CF02} used two sets of integers whose differences are not necessarily invertible, but such that the products of all differences in each of the sets are coprime with each other. This approach yields shares consisting of $\lfloor \log_2(n) \rfloor +2$ group elements each. ~\cite{CFS05} showed that these two approaches can be seen as special cases of a framework where the group $G$ is embedded into the tensor product $G\otimes_{\Z} R$, where $R=\Z[X]/(f)$ for an irreducible $f\in\Z[X]$, and then taking the evaluation points $\alpha_i$ in $R$. In fact, they show that, to obtain a threshold scheme, it is enough to choose the set of evaluation points so that the product of all $\alpha_i-\alpha_j$ has only $1$ and $-1$ as rational divisors (such a set is called primitive in \cite{CFS05}), and later show that this allows one to reduce the shares to $\lceil \log_2(n)\rceil$ group elements each.
Finally, \cite{CX20} presents an alternative approach applying a local-global principle, where they construct a family of threshold secret sharing schemes over every ring $\Z/{p^\ell}\Z$ (for every $p$ prime and natural $\ell\geq 1$) in such a way that they can be efficiently glued together to construct a black-box threshold secret sharing scheme. While the share size is slightly larger, namely $\lceil \log n\rceil+1$, than the result in \cite{CFS05}, this is an interesting conceptual alternative approach that is also generalized to quasi-threshold access structures (where rather than the secret being reconstructed by any set of $t+1$ parties, it can be reconstructed by any $t+c$ for some small integer $c>1$).

The best quantitative result in \cite{CFS05} achieves almost optimal share size when we want to share one element of the group: indeed, a known lower bound for the average share size for a $t$-threshold black-box secret sharing scheme (if $1\leq t\leq n-2$) is $\lfloor\log_2 n\rfloor-1$ group elements, by an argument presented in~\cite{CF02} which ultimately relies on \cite{KW93}. For $t=1$, a case that will be especially relevant below, one can prove the stronger lower bound $\log_2 n$~\cite{CF02}, which makes the construction in \cite{CFS05} optimal for the special case where $n$ is a power of $2$.

This leaves at least two questions open: an obvious one is whether we can close the gap between lower and upper bound for share size in the remaining cases; a second, more ``practical'' question is whether we can obtain better ``amortized'' share size in case we share a tuple of secrets from $G$, rather than a single one \footnote{Note the lower bound on the share-size above does not prevent this: while the bound could be applied to the direct product group $G'=G^m$, this would only preclude secret sharing schemes where the black-box access consists only of the group operation and inverse on $G'$, but it would not say anything about schemes where secrets are in $G^m$ but one also has access to operations on $G$}. The latter question is in part motivated by the framework described above: $G\otimes_{\Z} R$ is isomorphic as a $\Z$-module to $G^m$, and the secret sharing scheme can be seen as a ``Shamir sharing over $G^m$'' (with some restrictions on the evaluation points and on how to embed the secret). Unfortunately, though, using a primitive set of $R=\Z[X]/(f)$ as described above only allows us to obtain the desired threshold properties of the secret sharing schemes if we restrict ourselves to share secrets from $G\times \{0\}^{m-1}\subset G^m$ but will not, in general, allow us to reconstruct arbitrary secrets from $G^m$.

Suppose now that we have a set $S$ of $n$ matrices in $\textup{Mat}_m(\Z)$ and, for any natural number $t>0$, consider the following black-box secret sharing scheme for $n$ share receivers with secrets in $G^m$ and each share in $G^m$, inspired by the above framework: given a secret $s\in G^m$, choose a polynomial $f(X)$ uniformly at random in the set of polynomials in $G^m[X]$ of degree at most $t$ and such that the coefficient in $X^t$ is the secret $s$ (i.e. the remaining coefficients for $X^j$, $j=0,\dots, t-1$ are chosen uniformly at random in $G^m$) and define the shares to be $f(M)$ for each $M$ in $S$.
It turns out that the following holds:
\begin{itemize}
	\item For any $t>0$, if $S$ is a \textit{commutative} clique, the above construction is a $t$-threshold black-box secret sharing scheme. 
	\item For $t=1$, the above construction is a $1$-threshold black-box secret sharing scheme, \emph{even if $S$ is a non-commutative clique}.
\end{itemize} 

Therefore, the existence of exceptional cliques in $\textup{Mat}_m(\Z)$ of size $2^m$ would not only give a 1-threshold black-box secret sharing scheme with shares of size $m=\log_2 n$, matching the lower bound and known construction from \cite{CFS05} but it would also show that we can even ``pack'' a secret consisting of $m$ elements of the group, rather than a single one.
Moreover, the existence of a commutative exceptional clique of size $2^m$ would allow one to generalize these results to general threshold $t$.\\ 

Another application of exceptional units in cryptography can be found in the domain of zero knowledge proofs. More precisely, a $\Sigma$-protocol is perhaps the best known type of zero-knowledge proof. They are especially suitable for proving knowledge of a preimage of a linear map between vector spaces, without revealing this preimage. A paradigmatic example of $\Sigma$-protocol is the Schnorr proof of knowledge of discrete logarithm of a given element in a cyclic group of large prime order: given a group $G$ of prime order $q$, and public elements $g$ and $x$, the prover wants to convince the verifier that she knows $w\in\Z/q\Z$ such that $g^w=x$ without revealing $w$. The \textit{soundness} (the fact that a malicious prover can only convince the verifier of a false statement with small probability, in this case $1/q$) relies on the fact that the finite field is an exceptional clique.

Several works have studied how to extend $\Sigma$-protocols to other algebraic structures: \cite{CD09, CDK15} constructed $\Sigma$-protocols to prove knowledge of preimages for what they named ZK-ready functions: this is a class of maps between groups that in particular include cases where the groups involved are $A$-modules for $A$ a commutative ring with $1$, and the functions are group homomorphisms, although the notion is more general. On the other hand, \cite{BC24} defined a generalization of Schnorr and other $\Sigma$-protocols in terms of secret sharing schemes, showing that one can construct $\Sigma$-protocols for $A$-module homomorphisms from a linear secret sharing scheme for $A$-modules, where one can derive the needed security properties for the $\Sigma$-protocols from those of the underlying secret sharing scheme. In both cases, the ``soundness error'' of the protocol (the probability that a malicious prover successfully convinces a verifier of a false statement) is the inverse of the size of an exceptional clique (in $A$), which is used by the verifier during the protocol to challenge the prover. Rings with larger exceptional cliques will lead to smaller soundness errors.
Moreover, one can ``batch'' proofs (prove the knowledge of the preimages of several, say $m$, elements simultaneously, with better communication and computation than carrying out $m$ separate proofs) and in that case we are interested in having large exceptional cliques in $\textup{Mat}_m(A)$, or more generally in $\Hom_A(A^m, A^n)$ where $m\leq n$. In the latter case, the ratio $n/m$ roughly measures how much overhead in terms of computation and communication we will have, and we want to minimize $n$ with respect to m while at the same time having a large exceptional clique, with size as close to $2^m$ as possible. 

An especially relevant case is again when the statements to be proved are homomorphisms or ZK-ready functions defined on commutative groups of unknown order. Since these groups are $\Z$-modules, we are interested in analyzing the sizes of the largest exceptional cliques in $\textup{Mat}_m(\Z)$ and $\Hom_\Z(\Z^m,\Z^n)$ (as defined in Section~\ref{sec:rankmetric}). In this setting, \cite{CD09} proved that there is a clique of maximal possible size $2^m$ in $\Hom_\Z(\Z^m,\Z^{2m-1})$. In \cite{BC24}, this was generalized to show that if, for some $s>0$, $\textup{Mat}_s(\Z)$ contains a clique of size $2^s$, then for any multiple $s|m$, there is a clique of size $2^m$ in $\Hom_\Z(\Z^m,\Z^{2m-s})$; moreover, the existence of such optimal exceptional cliques in $\textup{Mat}_s(\Z)$ was shown for $s=3$ (while for $s=1,2$ the existence follows from elementary constructions). As shown in this work, also for $s=4$, $\textup{Mat}_s(\Z)$ has an exceptional clique of size $2^s$, while the question is open for all other $s>4$.

In summary, for this application to zero-knowledge proofs, we are also interested in studying the size that exceptional cliques in $\textup{Mat}_m(\Z)$ can have, and also in the more general case of exceptional cliques in $\Hom_\Z(\Z^m,\Z^n)$.

%In fact, we know that a lower bound for the average share size for a $t$-threshold black-box secret sharing scheme (if $1\leq t\leq n-2$) is $(\log_2 n)-1$ group elements, by a combination of arguments from \cite{CF02} and \cite{CCX13}, which ultimately rely on \cite{KW96} (for $t=1$, one can prove the stronger bound $\log_2 n$ \cite{CF02}).

%This was initiated by Desmedt and Frankl \cite{}, who used number fields to embed the group. However, this yields shares of average size $\Theta(n)$ which is far from the lower bound. Cramer and Fehr~\cite{CF02} and Cramer, Fehr and Stam~\cite{CFS04} later constructed black-box secret sharing scheme with shares of average size $\lfloor \log_2 n\rfloor +2$ and $\lceil \log_2 n \rceil$ respectively.
% TODO: Explain techniques
%  Acieving lower bound, packing

\section*{Acknowledgements}

The computational results of Theorem~\ref{thm:computational_commutative} were obtained using the computing resources from the Academic Leiden Interdisciplinary Cluster Environment (ALICE) provided by Leiden University.

\bibliographystyle{plain} 
\bibliography{citations}

\end{document}